\documentclass{amsart}[10pt,a4paper,twoside]

\usepackage[T1]{fontenc}
\usepackage[utf8]{inputenc}
\usepackage[english]{babel}
\usepackage{
  amsmath,
  amssymb,
  enumerate,
  bbm,
  textcomp,
  appendix,
  tabularx
}
\usepackage{tikz}
\usepackage[all]{xy}
\usepackage[active]{srcltx}
\usepackage{hyperref}
\usepackage{verbatim}
\usetikzlibrary{decorations.pathreplacing}

\theoremstyle{plain}
\newtheorem{thm}{Theorem}
\newtheorem*{lem*}{Lemma}
\newtheorem*{prop*}{Proposition}
\newtheorem*{cor*}{Corollary}

\renewcommand{\k}{\mathbbm{k}}
\newcommand{\tube}{one-parameter family of pairwise orthogonal tubes}
\newcommand{\za}{maximal convex family of $\mathbb ZA_\infty$
  components}
\newcommand{\hered}{hereditary abelian generating subcategory}
\newcommand{\hereds}{hereditary abelian generating subcategories}

\begin{document}

 \title[On the strong global dimension]{The strong global dimension of
  piecewise hereditary algebras}


\author{Edson Ribeiro Alvares}

\address[Edson Ribeiro Alvares]{Centro Polit\'ecnico, Departamento de Matem\'atica,
Universidade Federal do Paran\'a, CP019081, Jardim das Americas,
Curitiba-PR, 81531-990, Brazil
}

\thanks{The first named author acknowledges support from DMAT-UFPR and
CNPq-Universal 477880/2012-6}

\email{rolo1rolo@gmail.com}

\author{Patrick Le Meur}

\address[Patrick Le Meur]{
Laboratoire de Math\'ematiques, Universit\'e Blaise Pascal \&
  CNRS, Complexe Scientifique Les Cézeaux, BP 80026, 63171 Aubi\`ere
  cedex, France}

\curraddr{Universit\'e Paris Diderot, Sorbonne Paris Cit\'e, Institut de
   Math\'ematiques de Jussieu-Paris Rive Gauche, UMR 7586, CNRS,
  Sorbonne Universit\'es, UMPC Univ. Paris 06, F-75013, Paris, France}

\thanks{The second named author acknowledges financial support from
  FAPESP 2014/09310-5, CAPES, MathAmSud,
  and RFBM}

\email{patrick.lemeur@imj-prg.fr}

\author{Eduardo N. Marcos}

\address[Eduardo N. Marcos]{IME-USP (Departamento de Mat\'ematica),
  Rua Mat\~ao 1010 Cid. Univ., S\~ao Paulo, 055080-090, Brazil}

\thanks{The third named author acknowledges financial support from
  CNPq, FAPESP, MathAmSud, and Prosul-CNPq
  n\textdegree 490065/2010-4}

\email{enmarcos@ime.usp.br}

\dedicatory{In memory of Dieter Happel}







\date{\today}

\begin{abstract}
Let $A$ be a finite-dimensional piecewise hereditary algebra over an
algebraically closed field. This text
investigates the strong global dimension of $A$. This invariant is characterised in terms of the lengths of
sequences of tilting mutations relating $A$ to a hereditary abelian
category, in terms of the generating hereditary abelian subcategories
of the derived category of $A$, and in terms of the Auslander-Reiten
structure of that derived category.
\end{abstract}

\maketitle

\section*{Introduction}

Let $A$ be a finite-dimensional algebra over an algebraically closed
field $\k$. Its category of finitely generated (left) modules
is denoted by ${\rm mod}\,A$. Then, $A$ is called \emph{piecewise hereditary} if
the bounded derived category $\mathcal 
D^b({\rm mod}\,A)$ is equivalent as a triangulated category to
$\mathcal D^b(\mathcal H)$ where $\mathcal H$ is a hereditary abelian
($\k$-linear) category with split idempotents, finite-dimensional ${\rm
  Hom}$-spaces, and with tilting objects. In the particular case
where $A\simeq {\rm End}_{\mathcal H}(T)^{\rm op}$ for some tilting
object $T\in \mathcal H$, the algebra $A$ is called
\emph{quasi-tilted.} It is called \emph{tilted} when, in addition,
$\mathcal H\simeq {\rm mod}\,H$ for some finite-dimensional hereditary
algebra $H$. In \cite{MR1827736}, Happel proved that
a hereditary abelian category as above is equivalent to the category of
finitely generated modules over a hereditary algebra or to the 
category of coherent sheaves over a weighted projective line 
\cite{MR915180}.

In the representation theory of finite-dimensional algebras, piecewise
hereditary algebras play a particular and important role. On the one
hand, this is due to information that is already known on the
representation theory of hereditary 
algebras, of  tilted algebras (see \cite{MR675063}) or of quasi-tilted
algebras of canonical type (see \cite{MR1414820}), and also to
Happel's description of the bounded derived category of hereditary
abelian categories (see below). On the other
hand, these algebras are used in many parts of representation
theory. For instance, in order to develop the representation theory of
other classes of algebras such as the selfinjective algebras (see
\cite{MR2484737}) or the cluster tilted algebras (see
\cite{MR2409188}), in order to investigate singularity theory (see
\cite{MR3028577}), or in order to categorify cluster algebras (see
\cite{MR2249625}).

The homological
characterisation of quasi-tilted algebras \cite{MR1327209} and the
Liu-Skowro\'nski criterion for tilted algebras (see
\cite{MR2197389}) suggest that the quasi-tilted algebras are the
closest piecewise hereditary 
algebras to hereditary ones, and it is the main objective of this text
to give theoretical and numerical criteria to determine how far a
piecewise hereditary algebra is from being hereditary.

Recall the description of $\mathcal D^b(\mathcal H)$ made by Happel in
\cite{MR935124}: Any object is the direct sum of (finitely many) stalk
complexes $X[i]$ ($X\in \mathcal H$ and $i\in \mathbb Z$); And, for
every $i,j\in \mathbb Z$ and $X,Y\in \mathcal H$, the morphism space
${\rm Hom}(X[i],Y[j])$ is naturally isomorphic to ${\rm Hom}_{\mathcal
  H}(X,Y)$ if $i=0$, to ${\rm Ext}^1_{\mathcal H}(X,Y)$ if $i=1$, and
is equal to zero otherwise. Hence, when $\mathcal D^b({\rm
  mod}\,A)\simeq \mathcal D^b(\mathcal H)$, then there exists a tilting object $T\in\mathcal
D^b(\mathcal H)$ (that is, an object such that ${\rm Hom}(T,T[i])=0$
for $i\in\mathbb Z\backslash\{0\}$, and such that $\mathcal
D^b(\mathcal H)$ is the smallest full triangulated subcategory of
$\mathcal D^b(\mathcal H)$ containing $T$ and stable under taking
direct summands) such that $A\simeq {\rm End}(T)^{\rm op}$ as
$\k$-algebras. In particular, there exists   a
minimal $\ell\in\mathbb N$ and there exists $s\in\mathbb Z$ such
that $T$ lies in the additive closure $\bigvee_{i=0}^\ell\mathcal
H[s+i]$ of the union $\bigcup_{i=0}^\ell \mathcal
  H[s+i]$. When $\ell=0$, then $A$ is quasi-tilted. And one may
expect 
that the larger $\ell$, the further $A$ is from being quasi-tilted. 
Note however that there exist examples where $\ell=1$ and $A$ is
hereditary (see \cite{MR1648603}).

Recall also the characterisation proved by Happel, Rickard and
Schofield \cite{MR916069}: $\mathcal D^b({\rm mod}\,A)$ is equivalent
to bounded derived category of the module category of a
finite-dimensional hereditary  algebra if and only if there exists a
sequence of algebras with 
first term such a hereditary algebra, last term $A$, and where each term is isomorphic to
the opposite of the endomorphism algebra of a tilting module over the
preceding term. If the sequence has $\ell+2$ terms, then ${\rm
  gl.dim.}\,A\leqslant \ell+2$. Again, one expects that, the larger
$\ell$, the further is $A$ from being hereditary. However, in many
examples, ${\rm gl.dim.}\,A$ appears to be small whereas $\ell$ is
large.

In the  previous considerations, the parameter $\ell$  fails to
give a precise measure of how far  a piecewise hereditary algebra
is from being quasi-tilted. Recently a new invariant for piecewise
hereditary algebras has emerged and the present text aims at giving
some evidence of its relevance to give such a measure. This invariant
is the {\it strong global dimension}. It was first defined  in
\cite{S} in terms of  \emph{width} of complexes. The present 
text makes use of a slightly different definition given by Happel and
Zacharia in
\cite{MR2413349} and expressed in terms of \emph{length} of complexes. The
original definition may be recovered from the latter one by adding $1$
to the invariant.
Define the strong global dimension, ${\rm s.gl.dim.}\,A\in \mathbb
N\cup\{+\infty\}$, as follows.
 Let $X$ be an indecomposable object in the homotopy
  category of bounded complexes of finitely generated projective
  $A$-modules. Let 
  \[
    P\colon \cdots\to 0\to 0\to P^r \to P^{r+1}\to \cdots\to P^{s-1}\to P^s\to
    0\to 0\to \cdots\,,
  \]
be a minimal projective resolution of $X$,  where $P^r\neq 0$ and
$P^s\neq 0$.  Then define the length of $X$ as 
\[
  \ell(X)=s-r\,.
\]
The definition of the strong  global dimension used in this text is
\[
  {\rm s.gl.dim.}\, A=\underset{X}{{\rm sup}}\, \ell(X)
\]
where $X$ runs through all such indecomposable objects. It follows
from the definition that ${\rm s.gl.dim.}\,A=1$ if and only if $A$ is
hereditary and not semi-simple. In \cite[Problem 1]{MR2041672}, the
question was asked whether $A$ is piecewise hereditary when it has
finite strong global dimension. This has been studied by several
authors. The case of radical square-zero algebras was treated in 
\cite{MR2041672}.
Moreover, it follows from \cite{AS1,AS2} and \cite{S}
  (see also \cite{MR2041672}) that the class of algebras having both
  finite strong global dimension and a tame repetitive algebra
  coincides with the class of piecewise hereditary algebras having
  nonnegative Euler form.
The equivalence conjectured in \cite[Problem 1]{MR2041672} was proved in the general case
by Happel and Zacharia (\cite{MR2413349}) getting as a byproduct that
${\rm
  s.gl.dim.}\,A=2$ if and only if $A$ is
quasi-tilted and not hereditary. It is worth noticing the following
characterisation of the finiteness of the strong global dimension. By
the main result of \cite{S}, the algebra $A$ has finite strong global
dimension if and only if the repetitive algebra $\widehat A$ is
locally support finite in the sense of \cite{MR797444}. Hence, if
$A$ has finite strong global dimension, then the push-down functor
${\rm mod}\,\widehat A\to {\rm mod}\,T(A)$ is dense (here $T(A)$ denotes
the trivial extension). The question to know whether the converse
implication holds  was asked in \cite[Problem 2]{MR2041672}. So
far, that question remains open (see \cite[Theorem (B)]{AS2} for a
positive answer when $\widehat{A}$ is tame).

Let $\mathcal T$ be a triangulated category which is triangle
equivalent to the bounded derived category of a hereditary
abelian category (which is always assumed to be ${\rm Hom}$-finite,
to have split idempotents and tilting objects). Let $T\in\mathcal T$ be a tilting object and let $A$ be the
piecewise hereditary algebra ${\rm
  End}(T)^{\rm op}$. 
The purpose of this text is therefore to answer the following
questions:
\begin{itemize}
\item To what extend does
  ${\rm s.gl.dim.}\,A$ measure how far $A$ is from being
  quasi-tilted?
\item Is it possible to compute the strong global dimension or to
  characterise it?
\end{itemize}
The first main result of this text gives an answer in terms of Happel's
description of the bounded derived category of a hereditary abelian
category. Note that part (1) is a direct consequence of that description.
\begin{thm}
\label{thma}Let $\mathcal T$ be a triangulated category which is
triangle equivalent to the bounded derived category of a
hereditary abelian category.
Let $T\in\mathcal T$ be a tilting object. Assume that ${\rm
  End}(T)^{\rm op}$ is not a
  hereditary algebra. There exists a full and additive subcategory
  $\mathcal H\subseteq \mathcal T$ which is hereditary and abelian,
  such that the embedding $\mathcal H\hookrightarrow \mathcal T$
  extends to a triangle equivalence $\mathcal D^b(\mathcal H)\simeq
  \mathcal T$, and such that
  $T\in\bigvee_{i=0}^\ell\mathcal H[i]$
for some integer $\ell\geqslant 0$. Moreover
\begin{enumerate}
\item ${\rm s.gl.dim.}\,{\rm End}(T)^{\rm op}\leqslant\ell +2$ for any such
  pair $(\mathcal H,\ell)$, and
\item there exists such a pair $(\mathcal H,\ell)$ verifying ${\rm
    s.gl.dim.}\,{\rm End}(T)^{\rm op}=\ell +2$.
\end{enumerate}
\end{thm}

The second main result of this text gives an answer to the above
questions in terms of Happel, Rickard and Schofield's theorem and
using the concept of  \emph{tilting
mutations} in triangulated categories (\cite{MR2927802}): 
Let $T\in\mathcal T$ be a tilting object, let $T=T_1\oplus
T_2$ be a direct sum decomposition such that ${\rm
  Hom}(T_2,T_1)=0$ and consider  a triangle
$T_2'\to M\to T_2\to T_2'[1]$
where $M\to T_2$ is a minimal right ${\rm add}\,T_1$-approximation, then
 $T'=T_1\oplus T_2'$ is a tilting object and called
\emph{obtained from $T$ by tilting
mutation.} 
\begin{thm}
\label{thmb}Let $\mathcal T$ be a triangulated category which is
triangle equivalent to the bounded derived category of a
hereditary abelian category. 
  Let $T\in\mathcal T$ be a tilting object. Assume that ${\rm End}(T)^{\rm op}$ is not
  hereditary. Then there exists an integer $\ell\geqslant 0$ and a
  sequence $T^{(0)},T^{(1)},\ldots,T^{(\ell)}$ of tilting objects in
  $\mathcal T$ such that
  \begin{itemize}
  \item ${\rm End}(T^{(0)})^{\rm op}$ is a
    quasi-tilted algebra, $T^{(\ell)}=T$, and
  \item for every $i$ the object $T^{(i)}$ is obtained from
    $T^{(i-1)}$ by a tilting mutation.
  \end{itemize}
For any such sequence, ${\rm s.gl.dim.}\,{\rm End}(T)^{\rm
  op}\leqslant \ell+2$. Moreover, there exists such a sequence
such that  ${\rm s.gl.dim.}\,{\rm End}(T^{(i)})^{\rm
  op}=2+i$ for every 
  $i$ (and, in particular   ${\rm s.gl.dim.}\,{\rm
    End}(T)^{\rm op}=2+\ell$).
\end{thm}

This theorem is related to the second fundamental
result mentioned above 
in the following way. Let $A$ and $H$ be algebras such that $H$ is
hereditary and $\mathcal D^b({\rm mod}\,A)\simeq \mathcal D^b({\rm
  mod}\,H)$. Assume that $A_0=H,\ldots,A_{\ell+1}=A$ is a
sequence of algebras such that  $A_i={\rm
  End}_{A_{i-1}}(M^{(i-1)})^{\rm op}$ for a tilting $A_{i-1}$-module
  $M^{(i-1)}$ for every $i$. Then there exist tilting objects
  $T^{(0)},\ldots, T^{(\ell)}$ in $\mathcal D^b({\rm mod}\,H)$ such that
  $A_i\simeq {\rm End}(T^{(i-1)})^{\rm op}$ for every $i$, and which
  correspond to the tilting modules $M^{(0)},\ldots,M^{(\ell)}$ under
  suitable triangle equivalences $\mathcal D^b({\rm mod}\,H)\simeq
  \mathcal D^b({\rm mod}\,A_i)$. Then  ${\rm End}(T^{(0)})^{\rm op}$ is
  tilted, and it follows from \cite[Thm. 4.2]{MR2578597} that ${\rm
    s.gl.dim.}\,{\rm End}(T^{(i-1)})^{\rm op}\leqslant i+2$ for every
  $i$. When $H$ is of finite representation type the sequence
  $A_0,\ldots,A_{\ell+1}$ may be chosen such that $M^{(i)}$ is an APR
  tilting module for every $i$. In such a situation $T^{(i)}$ is
  obtained from $T^{(i-1)}$ by a tilting mutation. From this point of
  view, Theorem~\ref{thmb} expresses the strong global dimension as
  the infimum of the number $\ell+2$ of terms in  all possible
  sequences $A_0,\ldots,A_{\ell+1}$. 

The proofs of Theorem~\ref{thma} and Theorem~\ref{thmb} are based on the
description of ${\rm s.gl.dim.}\,{\rm End}(T)^{\rm op}$ in terms of
the connected  components of the Auslander-Reiten  quiver
 of $\mathcal
T$ in which  specific direct summands of $T$ lie. The
Auslander-Reiten structure of $\mathcal T$ is described in \cite{MR935124}.

\medskip

The strategy of these proofs
is described in Section~\ref{section_strat}. In particular, the
structure of the text is given in \ref{start5}.
The text uses the following notation and general setup.
Keep $\k$ and $\mathcal T$ as set
  previously.
Given a full subcategory
  $\mathcal H$ of $\mathcal T$ which is hereditary, abelian and stable
  under taking direct summands, the embedding $\mathcal
  H\hookrightarrow \mathcal T$ extends to an equivalence of
  triangulated categories $\mathcal D^b(\mathcal H)\xrightarrow{\sim}
  \mathcal T$ if and only if $\mathcal H$ generates $\mathcal T$ as a
  triangulated category. Subcategories satisfying all these conditions
  are used frequently in this text. They are called \emph{{\hereds}}.
When $\mathcal H$ arises from a weighted projective
  line, the full
subcategory consisting of torsion objects (or torsion free objects) is
denoted by $\mathcal H_0$ (or $\mathcal H_+$, respectively).
Given $X,Y\in \mathcal T$, the space ${\rm
  Hom}_{\mathcal T}(X,Y)$ is denoted by ${\rm Hom}(X,Y)$, and
the more convenient notation ${\rm
  Ext}^i(X,Y)$ will stand  for ${\rm Hom}(X,Y[i])$
($i\in\mathbb Z$) whenever $X$ and
  $Y$ lie in a same {\hered}.
Given an additive category $\mathcal A$, the
class of indecomposable objects in $\mathcal A$ is denoted by ${\rm
  ind}\,\mathcal A$. The standard duality functor ${\rm Hom}(-,\k)$ is
denoted by $D$.
By an {\it Auslander-Reiten component} (of
$\mathcal T$) is meant a connected component of the Auslander-Reiten
quiver of $\mathcal T$. The Auslander-Reiten translation of $\mathcal
T$ is denoted by $\tau$.
By a {\it transjective} component is meant an
Auslander-Reiten component  which has only finitely many
$\tau$-orbits.
  Let $\mathcal A\subseteq \mathcal T$ be a full additive subcategory
  stable under taking direct summands. 
It is called a  \emph{one-parameter
    family of
  pairwise orthogonal tubes} if it is convex and its indecomposable
objects form a (disjoint) 
union of pairwise orthogonal tubes in the Auslander-Reiten quiver of
$\mathcal T$, and if it is maximal for these properties;
equivalently, there exists a {\hered} $\mathcal
H\subseteq \mathcal T$ arising from a weighted projective line and
such that $\mathcal A=\mathcal H_0$ (see \ref{tube_family}). Here two
tubes $\mathcal U,\mathcal V$ are called \emph{orthogonal} whenever 
${\rm Hom}(X,Y[i])=0$ for every $X\in\mathcal U$, $Y\in \mathcal V$,
$i\in \mathbb Z$.
On the other hand, $\mathcal A$ is called a
\emph{{\za}} if it is  convex
and  its  indecomposable objects form a (disjoint) union of
Auslander-Reiten components of
shape $\mathbb ZA_{\infty}$ in the Auslander-Reiten quiver of
$\mathcal T$, and if it is maximal for these properties. This is the
case  if and only if exactly
one of the following situations occurs for some  {\hered} $\mathcal H\subseteq \mathcal
T$
\begin{itemize}
\item $\mathcal H$ arises from a weighted projective line with
  negative Euler characteristic and $\mathcal A=\mathcal H_+$,
  \item or else, $\mathcal H$ arises from a hereditary algebra of wild
    representation type and $\mathcal A$ consists of the objects
    obtained as direct sums of
    regular indecomposable
    modules over that algebra.
\end{itemize}

The reader is referred to \cite[Chap. XIII]{MR2360503} and
\cite[Chap. XVII]{MR2197389} for a general
account on the Auslander-Reiten structure (tubes, quasi-simples,
components of shape $\mathbb ZA_\infty$) of hereditary algebras of
tame and wild representation types, respectively. Note that in this
text all tubes are stable.

\section{Overview of the proofs of the main theorems and structure of the text}
\label{section_strat}

\subsection{An alternative definition for the strong global dimension}
\label{strat1}

The whole proof makes use of the following characterisation of $\rm
s.gl.dim.$ due to \cite[Lem. 5.6]{MR2862193}. Let $T, X  \in \mathcal
T$.  
Define
$\ell_T^+(X),\ell_T^-(X)\in\mathbb Z\cup\{-\infty,+\infty\}$ as follows
\[
\left\{
\begin{array}{l}
  \ell_T^+(X)={\rm sup}\,\{n\in\mathbb Z\ |\ {\rm Hom}(X,T[n])\neq 0
  \}\,,\\
  \ell_T^-(X)={\rm inf}\,\{n\in\mathbb Z\ |\ {\rm
    Hom}(T[n],X)\neq 0\}\,.
\end{array}\right.
\]

\begin{prop*}[\cite{MR2862193}]
 Let
  $T\in\mathcal T$ be a tilting object. Let $A={\rm
    End}(T)^{\rm op}$. Then $-\infty<\ell^-_T(X)\leqslant \ell_T^+(X)<+\infty$ for every
  $X\in \mathcal T$ indecomposable, and
${\rm s.gl.dim.}\,A={\rm sup}\,\{\ell_T^+(X)-\ell_T^-(X)\ |\
    X\in{\rm ind}\,\mathcal T\}$.
\end{prop*}

In the sequel, if $T$ is a tilting object in
$\mathcal T$,
and if $X\in{\rm ind}\,\mathcal T$, then
$\ell_T(X)$ denotes $\ell_T^+(X)-\ell_T^-(X)$. Note that
$\ell_T(X)=\ell_T(X[i])$ for every $X\in\mathcal T$ and $i\in\mathbb
Z$. Hence ${\rm s.gl.dim.}\,A={\rm sup}\,\{\ell_T(X)\ |\
    X\in{\rm ind}\,\mathcal T\,,\ \ell_T^-(X)=0\}$.

\subsection{Upper and lower bounds on the strong global
  dimension}
\label{upper_bound}\label{upperbound}

The starting point of the proof of Theorem~\ref{thma} is the following simple observation.  

\begin{prop*}
  Let $\mathcal H \subseteq \mathcal T$ be a hereditary abelian
  generating subcategory. Let
  $T\in\mathcal T$ be a tilting object. If
  $T\in\bigvee_{i=0}^\ell\mathcal H[i]$ for some $\ell \in \mathbb N$,
  then ${\rm s.gl.dim.}\,{\rm End}(T)^{\rm op}\leqslant \ell +2$.
\end{prop*}
\begin{proof}
  This follows directly from Happel's description of $\mathcal
  D^b(\mathcal H)$ and from the previous characterisation of the
  strong global dimension (see \ref{strat1}).
\end{proof}

Hence, the nontrivial part of Theorem~\ref{thma} is about lower bounds
on the strong global dimension. Proving that theorem amounts to finding
a {\hered} $\mathcal
H\subseteq\mathcal T$ such that $T\in \bigvee_{i=0}^\ell\mathcal H[i]$
and $\ell+2={\rm s.gl.dim.}\,{\rm End}(T)^{\rm op}$.
Should this hold true, whenever $X \in {\rm ind}\,{\mathcal T}$ is such that
$\ell_{T}(X) = {\rm s.gl.dim}\,{\rm End}(T)^{\rm op}$ and (up to shift)
$\ell^{-}_{T}(X) = 0$, then $X\in{\mathcal H}[\ell + 1]$,
and there should exist indecomposable summands $M_0, M_1$ of $T$ lying
respectively in $\mathcal H$ and ${\mathcal H}[\ell]$, together with
non-zero morphisms $M_1 \rightarrow X$ and $X \rightarrow
M_0[\ell + 2]$. In order to get lower bounds on ${\rm s.gl.dim.}\,{\rm
  End}(T)^{\rm op}$, one may therefore start from
$M_0$ and $M_1$ and look for $X$ and for nonzero morphisms $M_0\to X$
and $X\to M_0[\ell+2]$. Using the Auslander-Reiten structure of
$\mathcal T$, this is possible provided that $M_0$ and $M_1$ lie on
specific Auslander-Reiten components.

\subsection{The subcategories in which tilting objects start or end}
\label{strat4}

In view of the previous considerations, it is useful to have some
\emph{a priori} knowledge on the Auslander-Reiten components of
$\mathcal T$ containing indecomposable direct summands of $T$ lying in
$\mathcal H$ or $\mathcal H[\ell]$. In this text, these relevant
Auslander-Reiten components are determined in terms of 
subcategories in which $T$ {\it starts} or {\it ends.} Let $\mathcal
A\subseteq \mathcal T$ be a full, additive and convex subcategory stable
under taking direct summands. Here {\it convex} means that any path $X_0\to
\cdots \to X_n$ (of non-zero morphisms between indecomposable objects)
is contained in $\mathcal A$ as soon as $X_0,X_n\in\mathcal A$. In
this text $T$ is said to {\it start in $\mathcal A$} if the two
following conditions hold true
\begin{itemize}
\item $T$ has at least one indecomposable direct summand in $\mathcal
  A$,
\item for every indecomposable direct summand $X$ of $T$ there exists
  a path $X_0\to \cdots \to X_n$ such that $X_0\in\mathcal A$ and $X_n=X$.
\end{itemize}
Of course $T$ is said to {\it end in $\mathcal A$} if the dual
properties hold true.
There are obviously many
  such subcategories. This text concentrates on three specific
  ones,
namely
\begin{enumerate}[(1)]
\item when ${\rm ind}\,\mathcal A$ consists of a transjective Auslander-Reiten
  component (in which case $T$ is said to start in a transjective
  component, for short), or
\item when $\mathcal A$ is a one-parameter family of pairwise
  orthogonal tubes, or
\item when $\mathcal A$ is a {\za}.
\end{enumerate}

\subsection{Structure of the text}
\label{start5}

Consider a tilting object $T\in \mathcal
T$. Section~\ref{section_prel} establishes lower bounds on
the strong global dimension making hypotheses on the indecomposable
direct summands of $T$, as outlined in \ref{upper_bound}. For both
Theorem~\ref{thma} and Theorem~\ref{thmb}, the proofs reduce the
problem of determining ${\rm s.gl.dim.}\,{\rm End}(T)^{\rm op}$ by
replacing $T$ by an adequate tilting object obtained by tilting
mutation. In order to make a proper choice for that object,
Section~\ref{section_op} determines the subcategories in which a given
tilting mutation of $T$ may start or end. Once those subcategories are
known, it is necessary to determine the Auslander-Reiten components of
$\mathcal T$ that contain indecomposable direct summands of the
involved tilting objects. This task is performed in
Section~\ref{section_summandsAR}. Using all this information,
Section~\ref{section_ar} determines ${\rm s.gl.dim.}\,{\rm
  End}(T)^{\rm op}$ according to the subcategories in which $T$ starts
or ends. Finally, Section~\ref{section_pfs} proves Theorem~\ref{thma}
and Theorem~\ref{thmb}. Appendix~\ref{appendix_wpl} collects technical
material on bounded derived categories of weighted projective lines
needed for the proofs of the main theorems.

\section{Lower bounds on the strong global dimension}
\label{section_prel}

Given a tilting object $T\in \mathcal T$, this section gives  lower
bounds for ${\rm s.gl.dim.}\,{\rm 
  End}(T)^{\rm op}$. These form the technical heart of the proofs of
Theorem~\ref{thma} and Theorem~\ref{thmb}. They are
based on the existence of certain indecomposable direct summands of
$T$. A separate subsection is devoted to each one of these results
according to
the following situations: the considered summands lie in transjective
Auslander-Reiten components; or they lie in non-transjective 
Auslander-Reiten components; or they satisfy  specific vanishing assumptions on morphism
spaces.



\subsection{Lower bounds using transjective Auslander-Reiten
  components}
\label{lowerbound_section}

The first result on lower bounds on the strong global dimension is
based on the existence of certain
indecomposable direct summands of $T$ lying in transjective
Auslander-Reiten components. The corresponding setting is as follows.

\medskip

Let $\Gamma$ be a transjective Auslander-Reiten component of $\mathcal
T$. Let $\Sigma$ be a  slice of $\Gamma$. Let
  $S_1,\ldots,S_n$ be the
sources of $\Sigma$. Let $\mathcal
H\subseteq\mathcal T$ be the full subcategory
\[
  \mathcal H=\{X\in \mathcal T\ |\ (\forall S\in \Sigma)\ (\forall
  i\neq 0)\ {\rm Hom}(S,X[i])=0\}\,.
\]
Let $T\in\mathcal T$ be a tilting object and let $\ell\geqslant 1$ be
an integer such that
\begin{itemize}
\item $S_1,\ldots,S_n$ are all indecomposable
  summands of $T$,
\item there exists an indecomposable summand $L$ of $T$ lying in
  $\mathcal H[\ell]$.
\end{itemize}

\begin{lem*}
Under the previous setting there exists an object  $M\in\tau^{-1}\Sigma[\ell
+1]$ together with nonzero
morphisms
$L\to M$ and $M\to  \bigoplus_{i=1}^nS_i[\ell +2]$.
Hence $\ell_T(M)\geqslant \ell+2$. In particular, ${\rm
  s.gl.dim.}\,{\rm End}(T)^{\rm op}\geqslant \ell +2$.
\end{lem*}
\begin{proof}
  It is useful to prove first that $L\in\tau^{-2}\mathcal H[\ell]$. For
  this purpose note that
  \[
    \left({\rm ind}\,\mathcal H[\ell]\right)\backslash\left({\rm
        ind}\,\tau^{-2}\mathcal H[\ell]\right)=\Sigma[\ell]\cup
    \tau^{-1}\Sigma[\ell]\,,\ 
  \]
as sets of indecomposable objects. Since $L\in\mathcal H[\ell]$, the
claim therefore deals with 
proving that $L\not \in \Sigma[\ell]$ and $L\not\in
\tau^{-1}\Sigma[\ell]$. 
Using Serre duality and that $T$ is tilting
implies that
\[
  {\rm Hom}\left(\bigoplus_{i=1}^nS_i[\ell],L\right)=0\ \ {\rm and}\
    \ {\rm
      Hom}\left(\bigoplus_{i=1}^n\tau^{-1}S_i[\ell],L\right)=0\,.
\]
Since $S_1[\ell],\ldots,S_n[\ell]$ (or
$\tau^{-1}S_1[\ell],\ldots,\tau^{-1}S_n[\ell]$) are the sources of the
 slice $\Sigma[\ell]$ (or $\tau^{-1}\Sigma[\ell]$, respectively), this
entails that 
$L\not\in\Sigma[\ell]\ \ {\rm and}\ \ L\not\in\tau^{-1}\Sigma[\ell]$.
Thus
$L\in\tau^{-2}\mathcal H[\ell]$.

The category $\tau^{-2}\mathcal H[\ell]$ is abelian and its
indecomposable injectives are the objects in $\tau^{-1}\Sigma[\ell
+1]$, up to isomorphism. Hence
\[
  (\exists M\in\tau^{-1}\Sigma[\ell +1])\ \ {\rm Hom}(L,M)\neq 0\,.
\]
Besides ${\rm Hom}\left(\bigoplus_{i=1}^n\tau^{-1}S_i[\ell +1],M\right)\neq 0$ for
$\tau^{-1}S_i[\ell+1],\ldots,\tau^{-1}S_n[\ell+1]$ are the sources of
the  slice $\tau^{-1}\Sigma[\ell +1]$ of $\Gamma[\ell
+1]$. Serre duality then implies that
${\rm Hom}\left(M,\bigoplus_{i=1}^nS_i[\ell +2]\right)\neq 0$.
\end{proof}

\subsection{Lower bounds using non-transjective Auslander-Reiten
  components} 
\label{lowerbound_regular}

The second result on lower bounds on the strong global dimension is
based on the existence of certain
indecomposable direct summands of $T$ not lying in a transjective
Auslander-Reiten component. Here is the precise setting.

\medskip

Let $\mathcal H\subseteq \mathcal T$ be
{\hered}. Let $T\in \mathcal T$ be a
tilting object and $\ell\in \mathbb N$ be  such that there
exist indecomposable direct summands $M_0,M_1$ of $T$ satisfying the
following 
\begin{itemize}
\item $M_0\in\mathcal H$ and
\item   $M_1\in\mathcal H[\ell]$ and $M_1$ lies in a non-transjective
  Auslander-Reiten component of $\mathcal H[\ell]$. 
\end{itemize}
Recall that the
non-transjective Auslander-Reiten components are either tubes or of
 shape $\mathbb ZA_\infty$.
Let $X=\tau^{-1}M_0[\ell+1]$ and $Y=\tau
M_1[1]$. These lie in $\mathcal H[\ell+1]$.

\begin{lem*}
  Under the setting described 
  previously
  assume that
  both $M_0$ and 
  $M_1$ lie in non-transjective Auslander-Reiten components  of
  $\mathcal T$.
  \begin{enumerate}[(1)]
  \item If $\mathcal T$ contains a transjective Auslander-Reiten component then  ${\rm s.gl.dim.}\,{\rm End}(T)^{\rm op}\geqslant \ell +1$.
  \item If there exists a tube $\mathcal U\subseteq \mathcal T$ such
    that $M_0\in\mathcal U$ and $M_1\in \mathcal U[\ell]$ then ${\rm
      s.gl.dim.}\,{\rm End}(T)^{\rm op}\geqslant \ell+2$.
  \item If $\mathcal H$
    arises from a weighted projective line,
    $M_0\in
    \mathcal H_+$ and $M_1\in\mathcal H_0[\ell]$, then ${\rm
      s.gl.dim.}\,{\rm End}(T)^{\rm op}\geqslant \ell+2$.
  \item If  $M_0,M_1$ lie in Auslander-Reiten components of
    shape $\mathbb ZA_{\infty}$, then ${\rm s.gl.dim.}\,{\rm
      End}(T)^{\rm op}\geqslant
    \ell +2$.
  \end{enumerate}
\end{lem*}
\begin{proof}
  $(1)$ Let $\mathcal C$ be the Auslander-Reiten component of $\mathcal
  T$ such that $M_0[\ell +1]\in\mathcal C$. Let $\Gamma$ be the unique
  transjective Auslander-Reiten component of $\mathcal T$ such that
  \[
    (\forall V\in\mathcal C)\ (\exists U\in\Gamma)\ \ \ {\rm
      Hom}(U,V)\neq 0\,.
  \]
Let $\mathcal R$ be the disjoint union of the non-transjective
Auslander-Reiten components
such that
\[
  (\forall V\in\mathcal R)\ (\exists U\in \Gamma)\ \ \ {\rm
    Hom}(U,V)\neq 0\,.
\]
Therefore
\begin{itemize}
\item $\mathcal C\subseteq \mathcal R$,
\item $\mathcal R$ is the family of regular Auslander-Reiten
  components of $\mathcal H[\ell +1]$, and
\item $\mathcal R[-1]$ is the family of regular Auslander-Reiten
  components of $\mathcal H[\ell]$
and
  $M_1\in\mathcal R[-1]$.
\end{itemize}
In order to prove that ${\rm s.gl.dim.}\,{\rm End}(T)^{\rm op}\geqslant \ell+1$, it is
sufficient to find $S_0\in\Gamma$ such that
\[
  {\rm Hom}(M_1,S_0)\neq 0\ \ {\rm and}\ \ {\rm Hom}(S_0,M_0[\ell
  +1])\neq 0\,.
\]
First there exists a  slice $\Sigma$ in $\Gamma$ such that
\[
  (\forall S\in\Sigma)\ \ {\rm Hom}(S,M_0[\ell +1])\neq 0\,.
\]
Define
the full
subcategory $\mathcal H'\subseteq \mathcal T$ as
  $\mathcal H'=\{V\in\mathcal T\ |\ \text{${\rm
        Hom}(V,S[i])=0$ if $i\neq 0$ and $S\in \Sigma$}\}$.
Then
\begin{itemize}
\item $\mathcal H'$ is hereditary and abelian,
\item the indecomposable injectives of $\mathcal H'$ are the objects
  in $\Sigma$ up to isomorphism, and
\item $\mathcal R[-1]$ is the family of regular Auslander-Reiten
  components of $\mathcal H'$; in particular $M_1\in \mathcal H'$.
\end{itemize}
Thus
there exists $S_0\in\Sigma$ such that ${\rm Hom}(M_1,S_0)\neq
0$. By
hypothesis, ${\rm Hom}(S_0,M_0[\ell +1])\neq 0$.

\medskip

$(2)$ There exists a tube $\mathcal U\subseteq \mathcal T$ such that
$Y=\tau M_1[1]\in\mathcal U$ and $X=\tau^{-1}M_0[\ell+1]\in \mathcal U$.
Moreover there exist infinite sectional paths in $\mathcal U$
\[
  X\to \bullet\to \bullet\to \cdots\ \ {\rm and}\ \ \cdots
  \to \bullet\to \bullet\to Y\,.
\]
Since $\mathcal U$ is a tube, the two paths intersect. Hence there
exist $S\in\mathcal U$ and sectional paths in $\mathcal U$
\[
  \tau^{-1}M_0[\ell +1]\to \cdots\to S\ \ {\rm and}\ \ S\to \cdots
  \to \tau M_1[1]\,.
\]
Since the composition of morphisms along a sectional path does not
vanish, there exist nonzero morphisms
$\tau^{-1}M_0[\ell +1]\to S$ and $S\to\tau M_1[1]$.
Using Serre duality, this implies that
${\rm Hom}(S,M_0[\ell +2])\neq 0$ and ${\rm
    Hom}(M_1,S)\neq 0$.
Thus $\ell_T(S)\geqslant \ell +2$.

\medskip

$(3)$ Let $\mathcal U\subseteq\mathcal H_0[\ell+1]$ be the tube such
that $M_1[1]\in\mathcal U$. Applying \ref{morph} (part $(1)$) yields
an infinite sectional path $S=X_0\to \cdots\to X_n\to \cdots$ such
that
 ${\rm Hom}(\tau
^{-1}X_n, M_0[\ell+2])\neq 0$, for every $n\geqslant 0$. Since
$\mathcal U$ is a tube there also exists an infinite sectional path
$\cdots \to \bullet\to \bullet\to \tau^2 M_1[1]$, and the two paths
must intersect. Therefore there exists $n\geqslant 0$ together with a
sectional path $X_n\to \cdots\to \tau^2M_1[1]$. The composite morphism
$X_n\to \tau^2M_1[1]$ is thus nonzero. Using Serre duality entails
that ${\rm Hom}(M_1,\tau^{-1} X_n)\neq 0$. Thus
$\ell_T(\tau^{-1}X_n)\geqslant \ell+2$.

\medskip

$(4)$ The proof of the statement is better understood using the following diagram the
details of which are explained below and where all the arrows
represent nonzero morphisms.

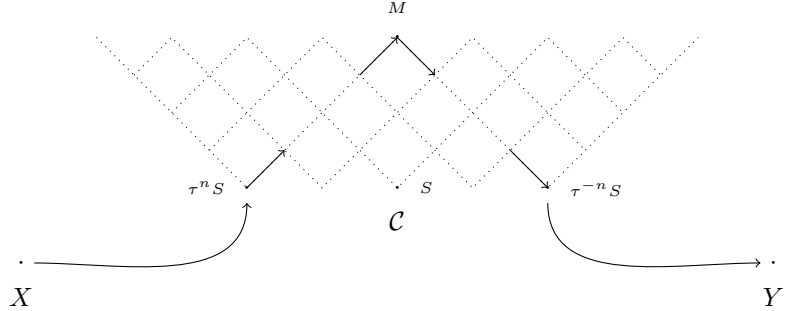
\begin{center}
  \begin{tikzpicture}[scale=0.5]
    \node (x) [label={below:$X$}] at (-6,-2) {$\cdot$};
    
    \node (y) [label={below:$Y$}] at (14,-2) {$\cdot$};

    \draw[dotted] (-4,4) -- (0,0) node (smn) [label={left:{\tiny$\tau^{n} S$}}]
    {$\cdot$} -- (4,4)  node (m) [label={above:{\tiny $M$}}] {$\cdot$} --
    (8,0) node (spn) [label={right:{\tiny $\tau^{-n} S$}}] {$\cdot$} -- (12,4);
    
    \draw[dotted] (-1,1) -- (0,2) -- (2,0) -- (6,4) -- (8,2) -- (9,1);
      
    \draw[dotted] (-3,3) -- (-2,4) -- (0,2) -- (2,4) -- (6,0) -- (10,4)
    -- (11,3) ;
      
    \draw[dotted] (-2,2) -- (0,4) -- (4,0) node (s) [label={right:{\tiny
        $S$}},label={below:$\mathcal C$}] {$\cdot$} -- (8,4) -- (10,2);
      
    \draw [->] (0,0) to (1,1);
      
    \draw [->] (3,3) to (4,4);
    
    \draw [->] (4,4) to (5,3);
      
    \draw [->] (7,1) to (8,0);
  
    \draw [->] (x) to [out=0,in=270]  (smn);
    
    \draw [->] (spn) to [out=270,in=180] (y);
    
  \end{tikzpicture}
\end{center}

Note that 
$X=\tau^{-1}M_0[\ell +1]$ and $Y=\tau M_1[1]$ lie in $\mathcal H[\ell+1]$.
\begin{itemize}
\item $S$ is any quasi-simple object in any Auslander-Reiten component
  of $\mathcal H[\ell+1]$ with shape $\mathbb ZA_{\infty}$ and $\mathcal C$ is
  that Auslander-Reiten component,
\item $n$ is any integer large enough such that ${\rm
    Hom}(X,\tau^nS)\neq 0$ 
  and ${\rm Hom}(\tau^{-n}S,Y)\neq 0$ (see \cite[Chap. XVIII,
  2.6]{MR2382332} or 
  \cite[Prop. 10.1]{MR1439198} according to whether $\mathcal H$
  arises from a hereditary algebra of wild representation type or
  from a weighted projective line with negative Euler characteristic, respectively), whence the curved arrows in the
  diagram, 
\item $M$ is the (unique) object in $\mathcal C$ such that there exist
  sectional paths of irreducible morphisms $\tau^n S\to \cdots \to M$
  and $M\to \cdots \to \tau^{-n}S$ in $\mathcal H[\ell+1]$. The arrows in the former path  (or, in the latter path) are all
  monomorphisms (or, epimorphisms, respectively) in $\mathcal H[\ell+1]$.
\end{itemize}
Since the diagram lies in
$\mathcal H[\ell+1]$, the composite morphisms $X\to M$
and $M\to Y$ arising from the paths $X\to \tau^n S\to \cdots \to M$
and $M\to \cdots \to\tau^{-n}S\to Y$ are nonzero. Using Serre
duality and the definition of $X$ and $Y$ yields: ${\rm
  Hom}(M_1,M)\neq 0$ and ${\rm Hom}(X,M_0[\ell+2])\neq 0$.
Thus, $\ell_T(M)\geqslant \ell+2$.
\end{proof}

\subsection{Lower bounds using non-vanishing morphism spaces}

 The
setting is the same as in  \ref{lowerbound_regular}. In particular the
same notation ($M_0,M_1,X,Y$) is used here.
Let $Z\to M_1$ be a
minimal right  almost split morphism in $\mathcal T$. As usual
$M_1$ is called {\it quasi-simple} 
whenever $Z$ is indecomposable.

\begin{lem*}
  Under the setting described previously, the following hold
  true.
  \begin{enumerate}[(1)]
\item If ${\rm
    Hom}(X,Y)\neq 0$, then $\ell_T(\tau^{-1}M_0)\geqslant \ell +2$. In
  particular ${\rm s.gl.dim.}\,{\rm End}(T)^{\rm op}\geqslant \ell +2$.
  \item If ${\rm Ext}^1(Y,X)\neq 0$, then $\ell_T(\tau Z)\geqslant
    \ell +2$ or $\ell_T(\tau^2
    M_1)\geqslant\ell +2$
    according to whether $M_1$ is quasi-simple or not.
In particular, ${\rm s.gl.dim.}\,{\rm End}(T)^{\rm op}\geqslant \ell +2$.
\item If ${\rm
    Hom}(X,Y)=0$ and ${\rm Hom}(Y,X)\neq 0$, then $\ell_T(\tau
  Z)\geqslant \ell +1$ or $\ell_T(\tau^2
    M_1)\geqslant \ell +1$ according to whether $M_1$ is quasi-simple
    or is not. In particular ${\rm s.gl.dim.}\,{\rm End}(T)^{\rm op}\geqslant \ell +1$.
\item If ${\rm
    Ext}^1(X,Y)\neq 0$, then $\ell_T(\tau M_1)\geqslant \ell +1$. In
  particular ${\rm s.gl.dim.}\,{\rm End}(T)^{\rm op}\geqslant \ell +1$.
\end{enumerate}
\end{lem*}
\begin{proof}
  (1) Serre duality gives
  \[
    \left\{
      \begin{array}{l}
        
0\neq{\rm Hom}(X,Y)={\rm Hom}(\tau^{-1}M_0[\ell+1],\tau M_1[1])\simeq D{\rm
  Hom}(M_1,\tau^{-1}M_0[\ell +1])\,,\\

0\neq {\rm Hom}(\tau^{-1}M_0[\ell +1],\tau^{-1}M_0[\ell +1])\simeq D{\rm
  Hom}(\tau^{-1}M_0[\ell +1],M_0[\ell +2])\,.

      \end{array}\right.
  \]
Thus $\ell_T(\tau^{-1}M_0[\ell +1])\geqslant \ell +2$.

\medskip

(2) The hypothesis implies that
$0\neq  {\rm Hom}(\tau^2
  M_1,M_0[\ell +1])$, and
hence
$\ell^+_T(\tau^2M_1)\geqslant \ell +1$.

Assume first that $M_1$ is not quasi-simple. Therefore ${\rm Hom}(\tau
M_1,M_1)\neq 0$. Serre duality then implies
${\rm
    Hom}(M_1[-1],\tau^2 M_1)\neq 0$.
Hence $\ell^-_T(\tau^2 M_1)\leqslant -1$, and thus
$\ell_T(\tau^2 M_1)\geqslant \ell +2$.

Assume now that $M_1$ is quasi-simple. Since $M_1$ lies in an 
Auslander-Reiten component of $\mathcal H[\ell]$ which is a tube or of
 shape $\mathbb ZA_{\infty}$, there exists an
almost split triangle $\tau^2 M_1\to \tau Z\to \tau M_1\to
\tau^2M_1[1]$.
Since $M_0[\ell +1]\in\mathcal H[\ell+1]$ and $\tau^2 M_1\in\mathcal
H[\ell]$,
it follows that
\[
  0\neq {\rm Ext}^1(Y,X)={\rm Hom}(\tau^2 M_1,M_0[\ell +1])\subseteq
  {\rm rad}(\tau^2 M_1,M_0[\ell +1])\,.
\]
In particular there exists a nonzero morphism $\tau^2 M_1\to
M_0[\ell +1]$ which factors through $\tau^2 M_1\to \tau Z$. Hence ${\rm
  Hom}(\tau Z,M_0[\ell +1])\neq 0$, and therefore
$\ell_T^+(\tau Z)\geqslant \ell +1$.
Moreover, using Serre duality yields
\[
  0\neq {\rm Hom}(\tau Z,\tau M_1)\simeq D {\rm Hom}(M_1[-1],\tau Z)\,.
\]
Hence $\ell_T^-(\tau Z)\leqslant -1$, and thus $\ell_T(\tau
Z)\geqslant \ell +2$.

\medskip

(3) The hypotheses imply that
$0\neq {\rm Hom}(Y,X)={\rm Hom}(\tau^2 M_1,M_0[\ell])$.
Hence
$\ell_T^+(\tau^2 M_1)\geqslant \ell$.

Assume first that $M_1$ is not quasi-simple. The argument used in (2)
to study $\ell_T^-(\tau^2 M_1)$
also applies here and shows that
$\ell_T^-(\tau^2 M_1)\leqslant -1$.
Thus $\ell _T(\tau^2 M_1)\geqslant \ell +1$.

Assume now that $M_1$ is quasi-simple. It follows from the hypotheses that
\[
  0={\rm Hom}(X,Y)={\rm Hom}(\tau^{-1}M_0[\ell +1],\tau M_1[1])={\rm
    Hom}(M_0[\ell ],\tau^2 M_1)\,.
\]
In particular $M_0[\ell]\not\simeq \tau^2 M_1$, and therefore (see above)
\[
  0\neq {\rm Hom}(Y,X)={\rm Hom}(\tau^2 M_1,M_0[\ell])\subseteq {\rm rad}(\tau^2
  M_1,M_0[\ell])\,.
\]
Hence there exists a nonzero morphism $\tau^2 M_1\to M_0[\ell]$ which
factors through $\tau^2 M_1\to \tau Z$. Therefore ${\rm
  Hom}(\tau Z,M_0[\ell])\neq 0$, and thus
$\ell_T^+(\tau Z)\geqslant \ell$.
The arguments used in (2) to prove that $\ell_T^-(\tau Z)\leqslant -1$
also apply here. Thus $\ell_T(\tau Z)\geqslant \ell +1$.

\medskip

(4) Serre duality gives
\[
  \left\{
    \begin{array}{l}
      0\neq {\rm Ext}^1(X,Y)={\rm Hom}(\tau^{-1}M_0[\ell +1],\tau
      M_1[2])\simeq D {\rm Hom}(\tau M_1, M_0[\ell])\,,\\
      0\neq {\rm Hom}(M_1,M_1)\simeq D{\rm Hom}(M_1[-1], \tau M_1)\,.
    \end{array}\right.
\]
Thus $\ell_T(\tau M_1)\geqslant \ell +1$.
\end{proof}

\section{Tilting mutations}
\label{section_op}

The proofs of Theorem~\ref{thma} and that of Theorem~\ref{thmb} use
inductions
based on
tilting
mutation. The inductive step produces a  new
tilting object $T'$ from the given tilting object $T$, such that ${\rm
  End}(T')^{\rm op}$ has strong
global dimension smaller than that of ${\rm End}(T)^{\rm op}$. 
This section therefore checks that tilting mutation permits a
convenient use of the lower bounds presented in
Section~\ref{section_prel}.
It
proceeds as follows.
\begin{itemize}
\item \ref{setting_approx} checks that $T'$ is indeed tilting.
\item \ref{sgldim_tilting} compares the lengths of a given object
  $X\in\mathcal T$ with respect to $T$ and $T'$. This leads to a
  comparison of the strong
  global dimensions of ${\rm End}(T)^{\rm op}$ and ${\rm End}(T')^{\rm
    op}$.
\item \ref{lem_approx} locates the indecomposable direct summands of
  $T'$ in terms of convex subcategories of $\mathcal T$ defined by
  indecomposable direct summands of $T$. This 
  serves to \ref{positions}.
\item In view of applying the results of
  Section~\ref{section_prel} to both $T$ and $T'$, \ref{positions}
  compares the subcategories in which $T$ and $T'$ start.
\end{itemize}

\subsection{Setting for the section}
\label{setting_approx}

 Let
$T\in\mathcal T$ be tilting. Consider a
direct sum decomposition
$T=T_1\oplus T_2$ such that ${\rm Hom}(T_2,T_1)=0$. Let $M\to T_2$ be a minimal right  ${\rm
  add}\,T_1$-approximation. It fits into a triangle
\begin{equation}
 \label{delta}
 T_2'\to M\to T_2\to T_2'[1]\,.\tag{$\Delta$}
\end{equation}
Let $T'=T_1\oplus T_2'$.
The following result is fundamental in this work. It is an application
of
\cite[Theorem 2.31 and Theorem 2.32 (b)]{MR2927802}
 since ${\rm
  add}(T_1)=\mu^-({\rm add}\,T,{\rm add}\,T_1)$ (with the notation
introduced therein).

\begin{prop*}
  Under the previous setting, $T'$ is a tilting object in $\mathcal T$.
\end{prop*}
\begin{proof}
Since the point of view and
notation in \cite{MR2927802} are slightly different from the ones in
this text a
proof is given below for the convenience of the reader.
  Because of \ref{delta}, the smallest triangulated subcategory of
  $\mathcal T$ containing $T'$ and stable under direct summands is
  $\mathcal T$. Hence is suffices to prove that ${\rm
    Hom}(T',T'[i])=0$ for every $i\neq 0$. Almost every argument below
  uses that $T$ is tilting so this will be implicit. Let $i\in\mathbb Z$.

First ${\rm Hom}(T_1,T_1[i])=0$ if $i\neq 0$ because $T_1\in{\rm
  add}(T)$.

Next there is an exact sequence obtained by applying ${\rm
  Hom}(T_1,-)$ to \ref{delta}
\[
  {\rm Hom}(T_1,M[i-1])\to {\rm Hom}(T_1,T_2[i-1])\to {\rm
    Hom}(T_1,T_2'[i])\to {\rm Hom}(T_1,M[i])\,.
\]
Since   $M\to T_2$ is an
${\rm add}\,T_1$-approximation it follows that
${\rm Hom}(T_1,T_2'[i])=0$ if $i\neq 0$.

Next there is an exact sequence obtained by applying ${\rm
  Hom}(-,T_1[i])$ to \ref{delta}
\[
  {\rm Hom}(M,T_1[i])\to {\rm Hom}(T_2',T_1[i])\to {\rm
    Hom}(T_2,T_1[i+1])\,.
\]
Since  
${\rm Hom}(T_2,T_1)=0$ it follows that
${\rm Hom}(T_2',T_1[i])=0$ if $i\neq 0$.

Next there is an exact sequence obtained by applying ${\rm
  Hom}(T_2,-)$ to \ref{delta}
\[
  \underset{\text{$=0$ if $i\neq 0$}}{\underbrace{{\rm
        Hom}(T_2,T_2[i])}}\to {\rm Hom}(T_2,T_2'[i+1])\to
  \underset{=0}{\underbrace{{\rm Hom}(T_2,M[i+1])}}
\]
where the rightmost term is zero if $i=-1$ by assumption on the
decomposition $T=T_1\oplus T_2$. Hence ${\rm Hom}(T_2,T_2'[i+1])=0$ if
$i\neq 0$. Therefore  the exact sequence obtained by
applying ${\rm Hom}(-,T_2')$ to (\ref{delta})
\[
  \underset{\text{$=0$ if $i\neq 0$ (see above)}}{\underbrace{{\rm Hom}(M,T_2'[i])}} \to
  {\rm Hom}(T_2',T_2'[i]) \to
  \underset{\text{$=0$ if $i\neq 0$}}{\underbrace{{\rm Hom}(T_2,T_2'[i+1])}}
\]
entails that ${\rm Hom}(T_2',T_2'[i])=0$ if $i\neq 0$.
All these considerations prove that $T'$ is tilting.
\end{proof}

\subsection{Behaviour of the strong global dimension under tilting
  mutations}
\label{sgldim_tilting}\label{sgldim_diff}

It is natural to ask for a relationship between $\ell_T(X)$ and
$\ell_{T'}(X)$. The following table gives an answer assuming (up to
suspension) that $\ell_{T'}^-(X)=0$, 
and denoting $\ell_{T'}(X)$ by $\ell$.

\begin{center}
\tiny
  \begin{tabularx}{\textwidth}{|>{\centering}X|>{\centering}X|>{\centering\arraybackslash}X|}
    
    \cline{2-3}

    \multicolumn{1}{X|}{} & & \\

    \multicolumn{1}{X|}{} & ${\rm
        Hom}(T_2,X[1])\neq 0$ & ${\rm
        Hom}(T_2,X[1])=0$\\
    
    \multicolumn{1}{X|}{} & & \\

    \hline

    & & \\

    ${\rm Hom}(X,T_1[\ell])\neq 0$ &  $\ell^-_{T}(X)=-1$,
    $\ell_{T}^+(X)=\ell$, and $\ell_{T}(X)=\ell+1$ &  $\ell^-_{T}(X)=0$,
    $\ell_{T}^+(X)=\ell$, and $\ell_{T}(X)=\ell$ \\

    & & \\

    \hline

    & & \\

     ${\rm Hom}(X,T_1[\ell])= 0$ &  $\ell^-_{T}(X)=-1$,
    $\ell_{T}^+(X)=\ell-1$, and $\ell_{T}(X)=\ell$ &  $\ell^-_{T}(X)=0$,
    $\ell_{T}^+(X)=\ell-1$, and $\ell_{T}(X)=\ell-1$ \\

    & & \\

    \hline

  \end{tabularx}
\normalsize
\end{center}

Indeed, examining the long exact
sequence obtained upon applying ${\rm Hom}(-,X)$ to $\Delta$ yields
that
\begin{itemize}
\item ${\rm Hom}(T_1[i],X)={\rm Hom}(M[i],X)=0$ if $i<0$ and ${\rm
    Hom}(T_2[i],X)=0$ if $i<-1$,
\item $\ell_T^-(X)\geqslant -1$, with an equality if and only if ${\rm
    Hom}(T_2[-1],X)\neq 0$,
\item if ${\rm Hom}(T_2[-1],X)=0$, then ${\rm Hom}(T_1,X)\neq 0$, and
  hence $\ell_T^-(X)=0$;
\end{itemize}
Using dual considerations yields that
\begin{itemize}
\item ${\rm Hom}(X, T_1[i])={\rm Hom}(X,M[i])= {\rm
    Hom}(X,T_2[i]) = 0$ if $i>\ell$,
\item $\ell_T^-(X)\leqslant \ell$, with an equality if and only if ${\rm
    Hom}(X,T_1[\ell])\neq 0$,
\item if ${\rm Hom}(X,T_1[\ell])=0$, then ${\rm Hom}(X,T_2[\ell-1])\neq 0$, and
  hence $\ell_T^+(X)=\ell-1$.
\end{itemize}

Using that table, the following is immediate (keeping the setting of \ref{setting_approx}).
\begin{prop*}
  $\vert {\rm s.gl.dim.}\,{\rm End}(T)^{\rm op}-{\rm s.gl.dim.}\,{\rm
      End}(T')^{\rm op}\vert\leqslant 1$.
  \end{prop*}

\subsection{Indecomposable direct summands and tilting mutation}
\label{lem_approx}

In view of comparing the subcategories of $\mathcal T$ in which $T$
starts or ends to the corresponding ones for $T'$, it is useful to
locate the indecomposable direct summands of $T'$ with respect to those of $T$. This is done in the following result.

\begin{lem*}
Let $X'\to N\to X\to X'[1]$ be a triangle in $\mathcal T$. The following conditions
are equivalent.
\begin{enumerate}[(i)]
\item $X$ is an indecomposable direct summand of $T_2$ and $N\to X$ is
  a right minimal ${\rm add}\,T_1$-approximation.
\item $X'$ is an indecomposable direct summand of $T_2'$ and $X'\to N$
  is a left minimal ${\rm add}\,T_1$-approximation.
\end{enumerate}
Moreover when these conditions are satisfied the object $X'$ lies in
the full and convex subcategory of $\mathcal T$ generated by $X[-1]$
and the indecomposable direct summands of $N$.
\end{lem*}
\begin{proof}
Assume $(i)$. Let $X\to T_2$ and $T_2\to X$ be 
morphisms  such that the composite morphism $X\to T_2\to X$ is
identity. Since $N\to X$ and $M\to T_2$ are ${\rm
  add}\,T_1$-approximations, there are commutative diagrams whose
rows are triangles
\[
  \xymatrix{
X' \ar[r] \ar[d] & N \ar[r] \ar[d] & X \ar[r] \ar[d] & X'[1] \ar[d]\\
T_2' \ar[r]  & M \ar[r] & T_2 \ar[r]  & T_2'[1]}
\ \ {\rm and}\ \ 
\xymatrix{
T_2' \ar[r] \ar[d] & M \ar[r] \ar[d] & T_2 \ar[r] \ar[d] & T_2'[1]
\ar[d]\\
X' \ar[r] & N \ar[r] &  X \ar[r] & X'[1]\,.
}
\]
Since $N\to X$ is right minimal and since the composite morphism $X\to 
T_2\to X$ is identity it follows that  the composite morphism
$N\to M\to N$ is an isomorphism, and hence so is  the composite
morphism $X'\to T_2'\to
X'$. This proves that $X'$ is a direct summand of $T_2'$.

\medskip

In order to prove that $X'$ is indecomposable, let $e\in {\rm End}(X')$
be an idempotent.
Since ${\rm
  Hom}(X[-1],N)=0$ ($T$ is tilting),
 there exist $f\in{\rm End}(N)$ and $g\in{\rm
  End}(X)$ making the following diagram commute
\[
  \xymatrix{
X' \ar[r] \ar[d]_e& N \ar[r] \ar[d]_f & X \ar[r] \ar[d]_g & X'[1]
\ar[d]_{e[1]}\\
X' \ar[r] & N \ar[r] & X\ar[r] & X'[1]\,.
}
\]
If $g$ is invertible then so is $f$ because $N\to X$ is a right
minimal ${\rm add}(T_1)$-approximation;
therefore $e$ is invertible, and hence $e=1_X$. If $g$ is
non-invertible then there exists $n\geqslant 1$ such that $g^n=0$
because $X$ is indecomposable; therefore the previous argument applies
to $(1-e^n=1-e,1-f^n,1-g^n=1)$ instead of to the triple $(e,f,g)$; it
entails that $e=0$. This proves that $X'$ is indecomposable.

\medskip

The functor ${\rm Hom}(-,T_1)$ applies to the triangle $X'\to N\to
X\to X'[1]$ and gives an exact sequence
\[
  {\rm Hom}(N,T_1)\to {\rm Hom}(X',T_1)\to {\rm Hom}(X[-1],T_1)
\]
where the rightmost term is zero because $T$ is tilting. Thus $X'\to
N$ is a left ${\rm add}\,T_1$-approximation.

\medskip

Finally let $u\in{\rm End}(N)$ be such that the composite
morphism $X'\to N\xrightarrow{u} N$ equals $X'\to N$. Therefore there
exists $v\in {\rm End}(X)$ such that the following diagram commutes
for every $n\geqslant 1$
\[
  \xymatrix{
X' \ar@{=}[d] \ar[r] & N \ar[r] \ar[d]_{u^n} & X \ar[r] \ar[d]_{v^n} &
X'[1] \ar@{=}[d]\\
X' \ar[r] & N \ar[r] & X \ar[r] & X'[1]\,.
}
\]
Since $N\in{\rm add}\,T_1$ and $X\in{\rm add}\,T_2$, the objects $N$
and $X\oplus X'$ are not isomorphic. Therefore, the triangle $X'\to
N\to X\to X'[1]$ does not split, and hence $X\to X'[1]$ is
nonzero. Therefore, $v^n\neq 0$ for
every $n\geqslant 1$. Since $X$ is indecomposable, it follows that
$v\colon X\to X$ is an isomorphism, and hence so is $u$. Thus,  $X'\to
N$ is left minimal. Therefore $(i)\Rightarrow (ii)$. 
The proof of the converse implication is obtained using dual considerations.

\medskip

Assume that $(i)$ and $(ii)$ hold true. Let $\mathcal C$ be the full
and convex subcategory of $\mathcal T$ generated by $X[-1]$ and the
indecomposable direct summands of $N$. Note that $X[-1]\to X'$ is
nonzero as observed earlier. If $N=0$ then $X'=X[-1]$, and hence
$X'\in \mathcal C$. If $N\neq 0$ then for every indecomposable direct
summand $Z$ of $N$ and for every retraction $N\twoheadrightarrow Z$
the composite morphism $X'\to N\twoheadrightarrow Z$ is nonzero
(\cite[Lemma 1.2, part $(ii)$]{MR2413349}). These morphisms together
with $X[-1]\to X'$ show that $X'\in\mathcal C$
\end{proof}

\subsection{Starting (or ending) subcategories and tilting mutation}
\label{positions}

In view of 
 Section~\ref{section_prel}, consider subcategories
 $\mathcal A,\mathcal B\subseteq
 \mathcal T$ such that $T$ starts in $\mathcal A\vee \mathcal B$ and
 ends in $\mathcal B[\ell]$ or in $\mathcal A[\ell]$ (for some
 $\ell$). A description of the subcategories in which $T'$ starts and
 ends is made
 \ref{positions_lem1}. And \ref{positions_lem2} and
 \ref{positions_lem3}  concentrate on the situations
 where $T$ starts in a {\tube} and
 ends in the suspension of it, or, more generally, in the $\ell$-th
 suspension of it, respectively.

\subsubsection{}
\label{positions_lem1}

  Let $\mathcal A,\mathcal B\subseteq \mathcal T$ be full, additive
  and convex   subcategories stable under taking direct summands and such that
  \begin{enumerate}[(a)]
  \item $\mathcal T=\bigvee_{i\in \mathbb Z}\left(\mathcal
      A[i]\vee\mathcal B[i]\right)$,
  \item ${\rm Hom}(\mathcal A,\mathcal A[i])=0$  and  ${\rm
      Hom}(\mathcal B,\mathcal B[i])=0$ if $i\neq 0,1$,
  \item ${\rm Hom}(\mathcal A,\mathcal B[i])=0$ if $i\neq 0$ and ${\rm
      Hom}(\mathcal B,\mathcal A[i])=0$ if $i\neq 1$. 
  \end{enumerate}
Typical examples used later on are as follows.
\begin{enumerate}[(1)]
\item $\mathcal A={0}$ and $\mathcal B=\mathcal H$ (where $\mathcal
  H\subseteq \mathcal T$
  is a {\hered}).
\item $\mathcal A=\mathcal H_+$ and $\mathcal B=\mathcal H_0$
  (assuming additionally that $\mathcal H$
arises from a weighted projective line).
\item $\mathcal A=\mathcal H_0[-1]$ and $\mathcal B=\mathcal H_+$
under the same additional assumption.
\item ${\rm ind}\,\mathcal A$ consists of the transjective
  Auslander-Reiten component of
  $\mathcal T$ intersecting $\mathcal H$, and ${\rm ind}\,\mathcal B=({\rm
    ind}\,\mathcal H\backslash {\rm ind}\,\mathcal A)$ (assuming
  additionally that
  $\mathcal H$ arises from a  hereditary algebra of wild
  representation type).
\end{enumerate}

\begin{lem*}
  Under the above setup let $\ell\geqslant 1$
  be an integer. Assume
  that $T$ starts in $\mathcal A\vee \mathcal B$ and ends in $\mathcal 
  B[\ell]$.
Let $T=T_1\oplus T_2$ be the direct sum decomposition such that $T_1\in \mathcal
A\vee \mathcal B\vee\cdots \vee \mathcal A[\ell -1]\vee \mathcal B[\ell
-1]$ and $T_2\in \mathcal A[\ell ]\vee \mathcal B[\ell]$.
Then $T'$ starts in $\mathcal A\vee \mathcal B$ and ends in $\mathcal
B[\ell -1]$, and $T$ and $T'$ have the same indecomposable direct
summands in $\mathcal A\vee \mathcal B$ when $l\geqslant 2$.
\end{lem*}
\begin{proof}
  Since $T'=T_1\oplus T_2'$ it suffices to prove that $T_2'\in
  \mathcal A[\ell-1]\vee \mathcal B[\ell-1]$ and that $T_2'$ has at
  least one indecomposable direct summand in $\mathcal B[\ell-1]$.

\medskip

Let $X'$ be an indecomposable direct summand of $T_2'$. Let $X'\to
N\to X\to X'[1]$ be a triangle such that $X'\to N$ is a left minimal
${\rm add}\,T_1$-approximation. Then $N\to X$ is a right minimal ${\rm
  add}\,T_1$-approximation, $X$ is an indecomposable direct
summand of $T_2$ and $X'$ lies in the full and convex subcategory of
$\mathcal T$ generated by $X[-1]$ and the indecomposable direct
summands of $N$ (\ref{lem_approx}). On the one hand $X[-1]\in \mathcal A[\ell-1]\vee
\mathcal B[\ell -1]$ because $X\in {\rm add}\,T_2$ and $T_2\in
\mathcal A[\ell]\vee \mathcal B[\ell]$. On the other hand
$N\in\mathcal A[\ell-1]\vee
\mathcal B[\ell-1]$ because $N\to X$ is a right minimal ${\rm
  add}\,T_1$-approximation and because of the assumptions made on
$\mathcal A$ and $\mathcal 
B$. Therefore $X'$ lies in the full and convex subcategory generated
by $\mathcal A[\ell -1]\vee \mathcal B[\ell -1]$ which already is full
and convex. Thus $X'\in \mathcal A[\ell -1]\vee \mathcal B[\ell
-1]$. This proves that $T_2'\in \mathcal A[\ell -1]\vee \mathcal
B[\ell -1]$.

\medskip

Let $X$ be an indecomposable direct summand of $T_2$ such that
$X\in\mathcal B[\ell]$. Let $X'\to N\to X\to X'[1]$ be a triangle such
that $N\to X$ is a right minimal ${\rm add}\,T_1$-approximation. Then
$X'$ is an indecomposable direct summand of $T_2'$ lying in the full
and convex subcategory of $\mathcal T$ generated by $X[-1]$ and the
indecomposable direct summands of $N$ (\ref{lem_approx}). Repeating
the arguments used
earlier to prove that $T_2'\in \mathcal A[\ell -1]\vee \mathcal B[\ell
-1]$ and using  that $X\in\mathcal B[\ell]$ entails that
$N\in\mathcal B[\ell-1]$ and hence $X'\in\mathcal B[\ell-1]$ (by
assumption, ${\rm add}(N)\cap \mathcal A[\ell]=\emptyset$). This
proves that $T_2'$ has at least one indecomposable direct summand in
$\mathcal B[\ell-1]$.
\end{proof}
\subsubsection{}
\label{positions_lem2}

The situation where there exists a {\hered}
$\mathcal H\subseteq \mathcal T$
arising from a weighted projective line
 and such that
$T$ starts in $\mathcal H_0$ and ends in  $\mathcal H_0[\ell]$ for some integer
$\ell$ needs careful consideration. Indeed the lower bound that is
relevant to this situation is \ref{lowerbound_regular}, part (2). It
requires  a tube in $\mathcal H_0$ containing an
indecomposable direct summand of $T$ and such that its $\ell$-th
suspension also contains an indecomposable direct summand of $T$. This
crucial fact is proved in \ref{positions_AR3} and
\ref{positions_AR4}. As a preparation, the
 following result  (when $\ell=1$) and the next one
 (when $\ell\geqslant 2$)  explain how these
   requirements are preserved
under tilting mutation.

\begin{lem*}
  Let $\mathcal H\subseteq \mathcal T$ be a {\hered}
  arising from a weighted projective line.
  Let $\mathcal U\subseteq \mathcal H_0$ be
  a tube. Assume that $T$ starts in $\mathcal H_0$ and ends in
  $\mathcal H_0[1]$, and that $T$ has at least one indecomposable
  direct summand in $\mathcal U$ and at least one indecomposable
  direct summand  in $\mathcal H_0[1]\backslash\mathcal U[1]$. Let
  $T=T_1\oplus T_2$ be the direct sum decomposition such that $T_2\in
  {\rm add}\,\mathcal
  U[1]$ and $T_1$ has no indecomposable direct summand in $\mathcal
  U[1]$. Then
  \begin{enumerate}[(1)]
  \item $T'$ starts in $\mathcal H_0$ and ends in $\mathcal H_0[1]$,
  \item for every tube $\mathcal V\subseteq \mathcal H_0$ there exists
    an indecomposable direct summand of $T'$ in $\mathcal V$ if and
    only if   there exists an
    indecomposable direct summand of $T$ in $\mathcal V$,
  \item for every tube $\mathcal V\subseteq \mathcal H_0$ there exists
    an indecomposable direct summand of $T'$ in $\mathcal V[1]$ if and
    only if $\mathcal V\neq \mathcal U$ and there exists an
    indecomposable direct summand of $T$ in $\mathcal V[1]$.
  \end{enumerate}
\end{lem*}
\begin{proof} 
Let $X'$ be an indecomposable direct summand of $T_2'$.  It is
sufficient to prove that  either $X'\in \mathcal U$ or else
  $X'\in\mathcal H_+[1]$.
Let $X'\to N\to
X\to X'[1]$ be a triangle such that $X'\to N$ is a left minimal ${\rm
  add}\,T_1$-approximation. Then $X$ is an indecomposable direct
summand of $T_2$, the morphism $N\to X$ is a right minimal ${\rm
  add}\,T_1$-approximation and $X'$ lies in the full and convex
subcategory of $\mathcal T$ generated by $X[-1]$ and the
indecomposable direct summands of $N$ (\ref{lem_approx}). In particular $X[-1]\in\mathcal
U$ because $T_2\in{\rm add}\,\mathcal U[1]$. Therefore in order to
prove that $X'\in\mathcal U$ or $X'\in\mathcal H_+[1]$ it is
sufficient to prove that $N\in \mathcal U\vee \mathcal H_+[1]$. Let
$Z$ be an indecomposable direct summand of $N$. Then ${\rm
  Hom}(Z,X)\neq 0$ because $N\to X$ is right minimal. Moreover $X\in
\mathcal U[1]$, $Z\in{\rm add}\,T_1$, the indecomposable direct
summands of $T_1$ lie either in $\mathcal H_0$, or in $\mathcal
H_+[1]$ or in $\mathcal H_0[1]\backslash\mathcal U[1]$ and, finally,
the tubes in 
  $\mathcal H_0$ are orthogonal.
This implies
 that $Z\in\mathcal U$ or $Z\in\mathcal H_+[1]$.
\end{proof}

\subsubsection{}
\label{positions_lem3}
\begin{lem*}
    Let $\mathcal H\subseteq \mathcal T$ be a {\hered}
    arising from a weighted projective
      line.
    Let $\ell\geqslant 2$ be an integer. Assume
  that $T$ starts in $\mathcal H_0$ and ends in $\mathcal H_0[\ell]$.
Let $T=T_1\oplus T_2$ be the direct sum decomposition such that
$T_1\in \mathcal
H_0\vee \mathcal H_+[1]\vee\cdots \vee\mathcal 
H_0[\ell-2]\vee \mathcal H_+[\ell-1]\vee \mathcal H_0[\ell-1]$ and $T_2\in
\mathcal H_+[\ell]\vee \mathcal H_0[\ell]$. Then
\begin{enumerate}[(1)]
\item $T'$ starts in $\mathcal H_0$ and ends in $\mathcal H_0[\ell-1]$,
\item $T$ and $T'$ have the same indecomposable direct summands in
  $\mathcal H_0$,
\item for every tube $\mathcal U\subseteq \mathcal H_0$, if $T$ has an
  indecomposable direct summand in $\mathcal U[\ell]$ then $T'$ has
  an indecomposable direct summand in $\mathcal U[\ell-1]$.
\end{enumerate}
\end{lem*}
\begin{proof}
$(1)$ and $(2)$ follow from \ref{positions_lem1} applied 
with $\mathcal A=\mathcal H_+$ and $\mathcal B=\mathcal H_0$.

\medskip

$(3)$ Let $\mathcal U\subseteq \mathcal H_0$ be a tube. Let $X$ be an
indecomposable direct summand of $T$ lying in $\mathcal U[\ell]$. In
particular $X$ is an indecomposable direct summand of $T_2$. Let
$X'\to N\to X\to X'[1]$ be a triangle such that $N\to X$ is a right
minimal ${\rm add}\,T_1$-approximation. It follows from \ref{lem_approx} that
$X'$ is an indecomposable direct summand of $T_2'$ and that $X'$ lies
in the full and convex subcategory of $\mathcal T$ generated by
$X[-1]$ and the indecomposable direct summands of $N$. In view of
proving $(3)$ it is therefore sufficient to prove that $N\in{\rm add}\,\mathcal
U[\ell -1]$. By the choice made for the decomposition $T=T_1\oplus
T_2$ and since $X\in \mathcal H_0[\ell]$ the  indecomposable
direct summands of $T_1$ from which there exists a nonzero morphism
to $X$ all lie in $\mathcal H_0[\ell-1]$. This forces $N\in\mathcal
H_0[\ell-1]$ because $N\to X$ is a right minimal ${\rm
  add}\,T_1$-approximation. Moreover since $X\in\mathcal U[\ell]$ and
since $\mathcal H_0$ consists of pairwise orthogonal tubes, the
indecomposable objects in $\mathcal H_0[\ell-1]$ from which there
exists a nonzero morphism to $X$ all lie in $\mathcal
U[\ell-1]$. Therefore $N\in {\rm add}\,\mathcal U[\ell-1]$.
\end{proof}

\section{Indecomposable direct summands of tilting objects in the
  Auslander-Reiten quiver}
\label{section_summandsAR}

Let $T\in\mathcal T$ be a tilting object. This section aims at giving
important information on the position of certain indecomposable direct
summands of $T$ in view of determining ${\rm s.gl.dim.}\,{\rm
  End}(T)^{\rm op}$. Recall that any Auslander-Reiten component of $\mathcal T$ has one
of the following shapes: transjective Auslander-Reiten component, tube or $\mathbb
ZA_{\infty}$. This section studies two particular situations, each of
which is studied in a separate subsection: when
$T$ starts in  a transjective component, and when $T$ starts in a {\tube}. In each one of
these situations a particular {\hered} of
$\mathcal T$ appears to be determined by $T$. This will be crucial to
prove Theorem~\ref{thma} and Theorem~\ref{thmb}.

\subsection{When $T$ starts in a transjective component}
\label{position_AR1}

\begin{prop*}
  Let $T\in\mathcal T$ be a tilting object and $\Gamma\subseteq
  \mathcal T$ be a transjective Auslander-Reiten component. Assume that $T$ starts in
  ${\rm add}\,\Gamma$. Then
   there exists a   slice $\Sigma$ in $\Gamma$ such
  that every source of $\Sigma$ is an indecomposable direct summand
  of $T$, and for every indecomposable direct summand $Y$ of $T$ lying in
  $\Gamma$ there exists a path in 
$\Gamma$ with source in $\Sigma$ and target $Y$.
\end{prop*}
\begin{proof}
 Since $T$ starts in ${\rm add}\,\Gamma$ there exist
   indecomposable summands $S_1,\ldots,S_n$ of $T$ lying in $\Gamma$
 such that ${\rm Hom}(\oplus_{i=1}^nS_i,X)\neq 0$ for every
  indecomposable direct summand $X$ of $T$ lying in $\Gamma$, and such
  that 
  ${\rm Hom}(S_i,S_j)=0$ if $i\neq j$. Let $\Sigma$ be the full
  subquiver of $\Gamma$ the vertices of which are those $X\in \Gamma$
  such that $X$ is the successor in $\Gamma$ of at least one of
   $S_1,\ldots,S_n$, and such that any path in $\Gamma$
  from any of $S_1,\ldots,S_n$ to $X$ is sectional.

By definition, $\Sigma$ is a convex subquiver of
$\Gamma$ intersecting each $\tau$-orbit at most once. Since $\Gamma$
is a transjective Auslander-Reiten component there exists
$n\in\mathbb Z$ such that $\tau^n X$ is a successor in $\Gamma$ of one
of the vertices in $\Sigma$, and $\tau^{n+1} X$ is the successor in
$\Gamma$ of none of the vertices in $\Sigma$. Consider
  any path in $\Gamma$
\begin{equation}
  S_i\to L_1\to L_2\to \cdots\to L_r=\tau^n X\tag{$\gamma$}
\end{equation}
 from one of $S_1,\ldots,S_n$ to $\tau^nX$. If
($\gamma$) were not sectional there would exist some hook
\[
  L_{t-1}\to L_t\to L_{t+1}=\tau^{-1} L_{t-1}\,,
\]
and hence a path in $\Gamma$
\[
  S_i\to L_1\to L_2\to \cdots\to L_{t-1}=\tau L_{t+1}\to \tau
  L_{t+2}\to \cdots \to \tau L_r=\tau^{n+1} X
\]
which would contradict the definition of $n$. The path ($\gamma$) is
therefore sectional. This proves that $\Sigma$ is a  slice in
$\Gamma$ fitting the requirements of the proposition.
\end{proof}

\subsection{When $T$ starts in a one-parameter family of pairwise orthogonal tubes}

Like when $T$ starts in a  transjective
 component of $\mathcal T$, there are relevant {\hereds} associated
with $T$ when it starts in a {\tube}. They arise from  weighted projective lines and there are two
cases to distinguish according to whether $T$ ends in (a suitable
suspension of) the subcategory of torsion objects or of torsion-free
objects, respectively. The
former case is dealt with in \ref{positions_AR3} and
\ref{positions_AR4}. The latter case is dealt with in
\ref{position_AR5}. In both cases it appears that $T$ cannot end in
a {\za} (\ref{positionAR2}).

\subsubsection{When $T$ starts in a {\tube} it cannot end in
a {\za}}
\label{positionAR2}

\begin{lem*}
  Let $\mathcal H\subseteq
  \mathcal T$ be a {\hered}
   equivalent to the
   category of coherent sheaves over a weighted projective line.
  Assume
  that $T$ starts in $\mathcal H_0$ and ends in $\mathcal H_+[\ell]$
  for some integer $\ell$. Then
  the weighted projective line has nonnegative Euler characteristic,
  equivalently $\mathcal H_+$ does not consist of $\mathbb ZA_\infty$
  components.
\end{lem*}
\begin{proof}
Note that it is necessary that $\ell\geqslant 1$ because $T$ starts in
$\mathcal H_0$. Thus $T\in\mathcal H_0\vee \mathcal H_+[1]\vee\mathcal
H_0[1]\vee \cdots \vee \mathcal H_+[\ell]$.
  Assume by contradiction that $\mathcal H_+$ consists of $\mathbb
  ZA_{\infty}$ components. A contradiction is obtained by induction on
  $\ell\geqslant 1$. 

\medskip

Assume that $\ell=1$. In particular $T\in\mathcal H_0\vee \mathcal
H_+[1]$ and $T$ has at least an indecomposable direct summand in $\mathcal H_0$
and  in $\mathcal H_+[1]$ respectively. Then ${\rm
  End}(T)^{\rm op}$ is not quasi-tilted  for, otherwise, there would
exist a {\hered} $\mathcal H'\subseteq \mathcal
T$ such that $T\in\mathcal H'$; since ${\rm End}(T)^{\rm op}$ is a
connected algebra there would therefore exist a tube $\mathcal
U\subseteq \mathcal H_0$ and a $\mathbb ZA_\infty$ component $\mathcal
V\subseteq \mathcal H_+[1]$ such that $\mathcal U,\mathcal V\subseteq
\mathcal H'$ and ${\rm Hom}(\mathcal U,\mathcal V)\neq 0$, which is
impossible. Then, it follows from \cite[Proposition 3.3]{MR2413349} and
\ref{upper_bound}
that ${\rm s.gl.dim.}\,{\rm End}(T)^{\rm op}=3$. The picture below
shows the subcategories of $\mathcal T$ containing indecomposable
direct summands of $T$ ($\circ$) and $T[3]$ ($\bullet$).

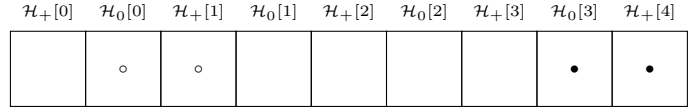
\tiny
\begin{center}
\begin{tikzpicture}

\foreach \k in {0,1,...,8}
{
\draw (\k,0) rectangle (\k + 1,-1);
}

\foreach \k in {0,1,...,4}
{
\draw (2*\k+0.5,0.25) node {$\mathcal H_+[\k]$};
}

\foreach \k in {0,1,...,3}
{
\draw (2*\k+1.5,0.25) node {$\mathcal H_0[\k ]$};
}

\draw (1.5,-0.5) node { $\circ$};

\draw (2.5,-0.5) node {$\circ$};

\draw (7.5,-0.5) node {$\bullet$};

\draw (8.5,-0.5) node {$\bullet$};

\end{tikzpicture}
\end{center}
\normalsize

Following in (\ref{eq:1}) in \ref{morph}, there is no indecomposable
$X\in\mathcal T$ such that both ${\rm Hom}(T,X)$ and ${\rm
  Hom}(X,T[3])$ vanish. This contradicts ${\rm s.gl.dim.}\,{\rm
  End}(T)^{\rm op}=3$.

\medskip

Now assume that $\ell\geqslant 2$. Let $T=T_1\oplus T_2$ be the direct
sum 
decomposition such that $T_1\in\mathcal H_0\vee \mathcal H_+[1]$ and
$T_2\in \mathcal H_0[1]\vee \mathcal H_+[2]\vee \cdots \vee \mathcal
H_+[\ell]$. Let $T_1\to M\to T_1'\to T[1]$ be a triangle where $T_1\to
M$ is a minimal left ${\rm add}\,T_2$-approximation. 
The dual versions of 
\ref{setting_approx} and \ref{positions_lem1} show
that $T_1'\oplus T_2$ is 
 tilting, lies in $\mathcal H_0[1]\vee \mathcal
H_+[2]\vee\cdots \vee\mathcal H_+[\ell]$ and has indecomposable direct
summands in $\mathcal H_0[1]$ and in $\mathcal H_+[\ell]$
respectively. This is impossible by the induction hypothesis.
\end{proof}

\subsubsection{}\label{positions_AR3}
When $\mathcal H\subseteq \mathcal T$ is a {\hered}
arising from a weighted projective
      line and
such that $T$
starts in $\mathcal H_0$ and ends in $\mathcal H_0[1]$, the following
lemma gives information on the indecomposable direct summands of $T$
lying in $\mathcal H_0$ or in $\mathcal H_0[1]$.

\begin{lem*}
  Let $\mathcal H\subseteq \mathcal T$ be a {\hered}
arising from a weighted projective
      line.
  Assume
  that $T$ starts in $\mathcal H_0$ and ends in $\mathcal
  H_0[1]$. Then
  \begin{enumerate}[(1)]
  \item there is no {\hered} of $\mathcal T$
    which contains $T$,
  \item ${\rm s.gl.dim.}\,{\rm End}(T)^{\rm op}=3$,
  \item there exists a tube $\mathcal U\subseteq \mathcal H_0$ such
    that
    \begin{enumerate}[(a)]
    \item every indecomposable direct summand of $T$ lying in
      $\mathcal H_0$ lies in $\mathcal U$,    
    \item every indecomposable direct summand of $T$ lying in
      $\mathcal H_0[1]$ lies in $\mathcal U[1]$.
    \end{enumerate}
  \end{enumerate}
\end{lem*}
\begin{proof}
  $(1)$ Proceed by absurd and assume that there exists a {\hered}
  $\mathcal H'\subseteq \mathcal T$ such that 
  $T\in\mathcal H'$. There are two cases to distinguish according to
  whether $\mathcal H'$ is equivalent to a module category or not.

\medskip

Assume first that $\mathcal H'$ is equivalent to ${\rm mod}\,H$ for
some finite-dimensional hereditary algebra $H$. Since $\mathcal
T\simeq \mathcal D^b(\mathcal H)\simeq \mathcal D^b(\mathcal H')$ it
follows that $H$ is of tame representation type. Since moreover $T$ has
indecomposable direct summands in $\mathcal H_0$ and $T\in\mathcal
H'={\rm mod}\,H$ it follows that ${\rm ind}\,\mathcal H_+$ consists of the
transjective Auslander-Reiten component of $\mathcal T$ containing the indecomposable
projective $H$-modules, and $\mathcal H_0$ consists of direct sums of
indecomposable regular $H$-modules. 

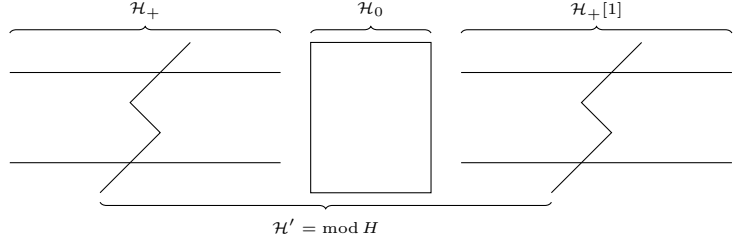
\tiny
\begin{center}
  \begin{tikzpicture}[scale=0.8]
    
    \draw (0,0.5) -- (4.5,0.5);

    \draw (0,2) -- (4.5,2);
    
    \draw (7.5,0.5) -- (12,0.5);

    \draw (7.5,2) -- (12,2);

    \draw (5,0) rectangle (7,2.5);

    \draw (1.5,0) -- (2.5,1) -- (2,1.5) -- (3,2.5);

    \draw (9,0) -- (10,1) -- (9.5,1.5) -- (10.5,2.5);

    \draw [decorate,decoration={brace,raise=0.125cm}] (5,2.5) --
    (7,2.5) node [pos=0.5,anchor=north,yshift=0.625cm] {$\mathcal H_0$};

    \draw [decorate,decoration={brace,raise=0.125cm}] (0,2.5) --
    (4.5,2.5) node [pos=0.5,anchor=north,yshift=0.625cm] {$\mathcal H_+$};

    \draw [decorate,decoration={brace,raise=0.125cm}] (7.5,2.5) --
    (12,2.5) node [pos=0.5,anchor=north,yshift=0.625cm] {$\mathcal H_+[1]$};

    \draw [decorate,decoration={brace,mirror,raise=0.125cm}] (1.5,0) --
    (9,0) node [pos=0.5,anchor=north,yshift=-0.25cm] {$\mathcal
      H'={\rm mod}\,H$};
  \end{tikzpicture}
\end{center}
\normalsize

Therefore $T\in\mathcal H'$
whereas it ends
in $\mathcal H_0[1]\subseteq
\mathcal H'[1]$. This is absurd.

\medskip

Assume next that $\mathcal H'$
arises from a weighted projective
      line.
Again there are
two cases to distinguish according to whether $\mathcal H_+'$ consists
of tubes or not.

If $\mathcal H_+'$ consists of tubes then $\mathcal H'$ arises from a
weighted projective line with vanishing Euler characteristic. Let
$\mathcal U\subseteq \mathcal H_0$ be a tube containing an
indecomposable direct summand of $T$. Then $\mathcal U\subseteq
\mathcal H'$ because $T\in\mathcal H'$. Therefore there exists
$q\in\mathbb Q\cup\{\infty\}$ such that  $\mathcal U\subseteq \mathcal
H'^{(q)}$. In other words $\mathcal U\subseteq \left(\mathcal
  H'\langle q\rangle \right)_0$ (see \ref{intr}). Applying \ref{suf} to $\mathcal H$
and $\mathcal H'\langle q\rangle$ entails that $\mathcal H=\mathcal
H'\langle q\rangle$, and hence $\mathcal H_0=\mathcal
H'^{(q)}$. Consequently $T\in\mathcal H'$ whereas $T$ has at least
one indecomposable direct summand in $\mathcal H_0[1]=\mathcal
H'^{(q)}[1]\subseteq \mathcal H'[1]$. This is absurd.

There only remains to treat the case where $\mathcal H_+$ does not
consist of tubes, and hence contains no tube. Since $\mathcal U$ is a
tube containing an indecomposable direct summand of $T$ and since
$T\in\mathcal H'$ it follows that $\mathcal U\subseteq \mathcal
H_0'$. Once again, applying \ref{suf} to $\mathcal H$ and $\mathcal
H'$ entails that $\mathcal H=\mathcal H'$. As observed previously this
leads to a contradiction since $T$ ends in $\mathcal H_0[1]$ and
$T\in\mathcal H'$.

\medskip

$(2)$ It follows from \ref{upper_bound} that ${\rm s.gl.dim.}\,{\rm
  End}(T)^{\rm op}\leqslant 3$. Moreover $(1)$ implies that ${\rm End}(T)^{\rm op}$ is not
quasi-tilted, and hence  ${\rm s.gl.dim.}\,{\rm
  End}(T)^{\rm op}\geqslant 3$ (\cite[Proposition 3.3]{MR2413349}).

\medskip

$(3)$ It is necessary to prove first that there exists a tube
$\mathcal U\subseteq \mathcal H_0$ such that both $\mathcal U$
and $\mathcal U[1]$ contain an indecomposable direct summand of
$T$. Since ${\rm s.gl.dim.}\,{\rm
  End}(T)^{\rm op}=3$, there exists an indecomposable $X\in\mathcal
T$ such that ${\rm Hom}(T,X)\neq 0$ and ${\rm Hom}(X,T[3])\neq
0$. Since $T\in\mathcal H_0\vee \mathcal H_+[1]\vee \mathcal H_0[1]$,
it follows that $X\in \mathcal H_0[2]$ and there exists indecomposable direct summands $Y,Z$ of
$T$ such that  $Z\in\mathcal H_0[1]$, ${\rm Hom}(Z,X)\neq 0$,
  $Y\in\mathcal H_0$ and ${\rm Hom}(X,Y[3])\neq 0$  (see
 \ref{morph} and picture
below where the other possible positions of the indecomposable direct
summands of $T$ or  $T[3]$ are marked with
$\circ$ or $\bullet$, respectively)

\tiny
\begin{center}
\begin{tikzpicture}

\foreach \k in {0,1,...,8}
{
\draw (\k,0) rectangle (\k + 1,-1);
}

\foreach \k in {0,1,...,4}
{
\draw (2*\k+0.5,0.25) node {$\mathcal H_0[\k]$};
}

\foreach \k in {1,2,...,4}
{
\draw (2*\k-0.5,0.25) node {$\mathcal H_+[\k ]$};
}

\draw (0.5,-0.5) node {$\circ$};

\draw (1.5,-0.5) node {$\circ$};

\draw (7.5,-0.5) node{$\bullet$};

\draw (8.5,-0.5) node {$\bullet$};

\draw (2.5,-0.5) node { $Z$};

\draw (4.5,-0.5) node {$X$};

\draw (6.5,-0.5) node {$Y[3]$};

\end{tikzpicture}
\end{center}
\normalsize

Since $\mathcal H_0$ consists of pairwise orthogonal tubes, there
exists a tube $\mathcal U\subseteq \mathcal H_0$ such that
$Z\in\mathcal U[1]$, $X\in\mathcal U[2]$ and $Y[3]\in\mathcal
U[3]$. In particular $Y\in\mathcal U$ and $Z\in\mathcal U[1]$. This
proves the claim. In other words if $\mathcal E(T)$ denotes the set of
those tubes $\mathcal U\subseteq \mathcal H_0$ such that each one of
$\mathcal U$ and $\mathcal U[1]$ contains an indecomposable direct
summand of $T$ then $\mathcal E(T)\neq \emptyset$.

\medskip

Next it useful to prove that $\mathcal E(T)$ consists of a single tube
which contains all indecomposable direct summands of $T$ lying in
$\mathcal H_0$. Applying \ref{positions_lem2} to $T$ and repeating the
application for every tube lying in $\mathcal E(T)\backslash\{\mathcal
U\}$ eventually yields
a tilting object $S\in\mathcal T$ such that
\begin{itemize}
\item $S$ starts in $\mathcal H_0$ and ends in $\mathcal H_0[1]$,
\item $\mathcal E(S)$ consists of a single tube $\mathcal U$,
\item every indecomposable direct summand of $S$ lying in $\mathcal
  H_0[1]$ lies in $\mathcal U[1]$,
\item for every tube $\mathcal V\subseteq\mathcal H_0$ there exists an
  indecomposable direct summand of $T$ lying in $\mathcal V$ if and
  only if the same holds true for $S$.
\end{itemize}
It is not possible for $S$ to have any indecomposable direct summand
lying in $\mathcal H_0\backslash\mathcal U$ for, otherwise, the dual
version of \ref{positions_lem2} could  apply to $S$ and 
  $\mathcal
U[1]$  and yield a tilting
object $S'\in\mathcal T$ starting in $\mathcal H_0$, ending in
$\mathcal H_0[1]$ and such that $\mathcal E(S')=\emptyset$. Therefore
every  indecomposable direct summand of $S$ lying in $\mathcal H_0$ (and
hence every indecomposable direct summand of $T$ lying in $\mathcal
H_0$) lies in $\mathcal U$. In particular $\mathcal E(T)=\{\mathcal
U\}$.

\medskip

Finally it is not possible for $T$ to have any indecomposable direct
summand in $\mathcal H_0[1]\backslash\mathcal U[1]$ for, otherwise,
\ref{positions_lem2} could apply to $T$ and $\mathcal U$,  and yield a tilting object
$T'\in\mathcal T$ starting in $\mathcal H_0$, ending in $\mathcal
H_0[1]$ and such that $\mathcal E(T')=\emptyset$. This proves $(3)$.
\end{proof}

\subsubsection{}\label{positions_AR4}
The
previous result extends
as follows when
$T$ starts in $\mathcal H_0$ and ends in $\mathcal H_0[\ell]$ for a {\hered}
$\mathcal H\subseteq \mathcal T$
arising from a weighted projective
      line and for  $\ell\geqslant 1$.

\begin{prop*}
  Let $\mathcal H\subseteq\mathcal T$ be a {\hered}
arising from a weighted projective
      line.
  Assume that $T$ starts in $\mathcal H_0$
  and ends in $\mathcal H_0[\ell]$ for some integer $\ell\geqslant
  0$.
Then there exists a tube $\mathcal U\subseteq\mathcal H_0$ such that 
\begin{enumerate}[(a)]
\item $\mathcal U$ contains every indecomposable direct summand of
  $T$ lying in $\mathcal H_0$,
\item $\mathcal U[\ell]$ contains every indecomposable direct summand
  of $T$ lying in $\mathcal H_0[\ell]$.
\end{enumerate}
In particular, if $\ell =0$ then $\mathcal U$ contains every
indecomposable direct summand of $T$.
\end{prop*}
\begin{proof}
When $\ell=0$
  the hypotheses entail that $T\in\mathcal H_0$. The conclusion then
follows from the fact that  $\mathcal H_0$
consists of pairwise orthogonal tubes.

\medskip

  The general case proceeds by induction on $\ell\geqslant 1$. The case $\ell=1$ is
  dealt with using \ref{positions_AR3} so assume that $\ell\geqslant 2$. Let $T'\in\mathcal T$ be the
  tilting object introduced in \ref{positions_lem3}. Therefore the
  induction hypothesis applies to $T'$. Let $\mathcal
  U\subseteq\mathcal H_0$ be the tube such that $\mathcal U$ (or
  $\mathcal U[\ell]$) contains every indecomposable direct summand of
  $T'$ lying in $\mathcal H_0$ (or in $\mathcal H_0[\ell-1]$,
  respectively). Then, it follows from \ref{positions_lem3}, parts
  $(2)$ and $(3)$, and from the fact that $T$ ends in $\mathcal
  H_0[\ell]$ that the conclusion of the proposition holds true
  for $T$
\end{proof}

\subsubsection{}
\label{position_AR5}
When $T$ starts in a {\tube} 
but does not end in a suitable suspension of that family (unlike
\ref{positions_AR4}), then  there still exists  a relevant {\hered} associated with
$T$ as explained in the following 
result. 

\begin{prop*}
  If
  $T$ starts in a {\tube} and does not end in  a transjective
  component, then
  there exists a {\hered} $\mathcal H\subseteq
  \mathcal T$
arising from a weighted projective
      line
  and there exists an integer $\ell\geqslant
  0$ such that
  \begin{enumerate}[(a)]
  \item $T$ starts in a {\tube} contained in $\mathcal H$, and
  \item $T$ ends in $\mathcal H_0[\ell]$.
  \end{enumerate}
\end{prop*}
\begin{proof}
  It follows from \ref{positionAR2} that $T$ ends in {\tube}. Moreover \ref{tube_family} shows
  that there exists a {\hered} $\mathcal
  H'\subseteq \mathcal T$
  arising from a weighted projective
      line and
  such that $\mathcal H_0'$ is that
  family. Let $\ell\geqslant 0$ be the integer such that $T$ starts in
  $\mathcal H'[-\ell]$. Let $\mathcal H=\mathcal H'[-\ell]$. Then
  $\mathcal H$ fits the conclusion of the proposition.
\end{proof}

\section{The strong global dimension through Auslander-Reiten theory}
\label{section_ar}

Let $T\in\mathcal T$ be a tilting object. The objective of this section
is to determine ${\rm s.gl.dim.}\,{\rm End}(T)^{\rm op}$ in terms of the
position of the indecomposable direct summands of $T$ in the
Auslander-Reiten quiver of $T$.
Recall that $T$ may start either in  a transjective
component, or in a {\tube},
or in a {\za}.
 The  
three following subsections therefore treat each one of these cases
separately. The situation where $T$ {\it ends}
in a transjective component is dual to that where $T$ starts in a
transjective component. Thus when assuming that $T$ {\it does not
  start} in a transjective component it may be assumed also that $T$
does not end in a transjective component.

\subsection{When $T$ starts in a transjective component}
\label{transjective}

\begin{prop*}
  Let $T\in\mathcal T$ be a tilting object. Assume that $T$
  starts in  a transjective component $\Gamma$.  Let
  $\Sigma$ be the slice introduced in \ref{position_AR1}.
\begin{enumerate}[(1)]
\item Let
  $\mathcal H=\{X\in\mathcal T\ |\ (\forall S\in \Sigma)\ 
  (\forall i\neq 0)\ \ {\rm Hom}(S,X[i])=0\}$. Then $\mathcal H$ is a
  {\hered}. Moreover there exists an
  integer $\ell\geqslant 0$ such that $T\in\bigvee_{i=0}^\ell\mathcal H[i]$
and such that $T$ has an indecomposable summand in $\mathcal H$ and in
$\mathcal H[\ell]$;
\item If  ${\rm End}(T)^{\rm op}$ is not a hereditary algebra then
  ${\rm s.gl.dim.}\,{\rm End}(T)^{\rm op}=\ell +2$.
\end{enumerate}
\end{prop*}
\begin{proof}
(1) The first assertion follows from the fact that $\Sigma$ is a
slice in $\Gamma$. In particular $\mathcal T=\bigvee_{i\in \mathbb
  Z}\mathcal H[i]$. The second assertion follows from the following
facts: the indecomposable projectives in $\mathcal H$ are, up to
isomorphism, the objects in $\Sigma$; and the sources of $\Sigma$
are all summands of $T$.

\medskip

(2) If $\ell=0$ then
${\rm End}(T)^{\rm
  op}$ is 
quasi-tilted and not
hereditary. Thus
${\rm
  s.gl.dim.}\,{\rm End}(T)^{\rm op}=2$ (\cite[Prop. 3.3]{MR2413349}). If $\ell\geqslant 1$
then the conclusion follows from \ref{upperbound} and \ref{lowerbound_section}.
\end{proof}

\subsection{When $T$ starts and ends in
a {\tube}}
\label{tame_sgldim}

\begin{prop*}
Assume that $T$ starts in a {\tube} and does not end in  a transjective
component. Let
$\mathcal H,\ell$ be like in \ref{position_AR5}. Then ${\rm
  s.gl.dim.}\,{\rm End}(T)^{{\rm op}}=\ell +2$.
\end{prop*}
\begin{proof}
Assume first that $T$ starts in $\mathcal H_0$. Let $\mathcal
U\subseteq\mathcal H_0$ be the
  tube obtained upon applying \ref{positions_AR4} to $T$.
Since $T$ starts in $\mathcal H_0$ and ends in $\mathcal H_0[\ell]$ it
follows that $T\in\mathcal H\vee\mathcal H[1]\vee\cdots\vee\mathcal
H[\ell]$. Therefore ${\rm s.gl.dim.}\,{\rm End}(T)^{{\rm op}}\leqslant
\ell+2$ (\ref{upper_bound}). On the other hand it follows from
\ref{lowerbound_regular} (part $(2)$) that ${\rm s.gl.dim.}\,{\rm
  End}(T)^{\rm op}\geqslant \ell+2$.

If $T$ does not start in $\mathcal H_0$ then it starts in $\mathcal
H_+$. The arguments used in the previous case  lead to the same
conclusion provided that \ref{lowerbound_regular}, part $(3)$, is used
instead of \ref{lowerbound_regular}, part $(2)$.
\end{proof}

\subsection{When $T$ starts  in  a {\za}}
\label{wild_sgldim}

\begin{prop*}
  Assume that $T$ starts  in a {\za} and does not end
  in a transjective component. Let 
  $\mathcal H\subseteq\mathcal T$ be a {\hered}
  such that $T$ starts in $\mathcal H$. Let $\ell\geqslant 0$ be such
  that $T$ ends in $\mathcal H[\ell]$. Then ${\rm s.gl.dim.}\,{\rm
    End}(T)^{\rm op}=\ell+2$.
\end{prop*}
\begin{proof}
First,
${\rm
  s.gl.dim.}\,{\rm End}(T)^{\rm op}\leqslant \ell+2$
due to \ref{upperbound}.
 Let
$M_0\in\mathcal H$ be an
indecomposable direct summand of $T$ lying in an Auslander-Reiten
component of shape $\mathbb ZA_{\infty}$. There are two
cases to distinguish according to whether or not there exists an
indecomposable direct summand $M_1$ of $T$ lying in an Auslander-Reiten
component of shape $\mathbb ZA_{\infty}$  contained in $\mathcal
H[\ell]$. Assume first that this is indeed the case.
According to
\ref{lowerbound_regular} (part $(4)$) it
follows that ${\rm s.gl.dim.}\,{\rm End}(T)^{\rm op}\geqslant \ell+2$,
and hence ${\rm s.gl.dim.}\,{\rm End}(T)^{\rm op}=\ell+2$. 

Otherwise, it is necessary that $\mathcal H$ arises from a weighted
projective line and that $T$ ends in $\mathcal H_0[\ell]$.
Let $M_1$ be an indecomposable direct summand
of $T$ lying in $\mathcal H_0[\ell]$.
 According to \ref{lowerbound_regular}
(part $(3)$) it follows that ${\rm s.gl.dim.}\,{\rm End}(T)^{\rm op}\geqslant \ell+2$,
and hence ${\rm s.gl.dim.}\,{\rm End}(T)^{\rm op}=\ell+2$. 
\end{proof}

\section{Proofs of the main theorems}
\label{section_pfs}

It is worth noticing that if a statement holds true for tilting
objects starting in  a 
transjective component then so does its dual statement for tilting
objects ending in a transjective component. Hence all the
possible situations (up to dualising) for a tilting object $T$ are covered by the three
following cases:
\begin{enumerate}[(a)]
\item $T$ starts in a transjective component, or
\item $T$
  starts in a {\tube} and does not end in a transjective component, or
\item $T$ starts in a {\za} and does not end in a
  transjective component.
\end{enumerate}

\subsection{Proof of Theorem~\ref{thma}}
\label{thma_proof}

\begin{proof}
Assertion $(1)$ follows from \ref{upperbound}. Assertion $(2)$ follows
from \ref{transjective}, \ref{tame_sgldim} and \ref{wild_sgldim}.

\end{proof}

\subsection{Proof of Theorem~\ref{thmb}}

\begin{proof}
Only the last assertion needs a
proof (see \ref{sgldim_diff}). Proceed by induction on $d={\rm s.gl.dim.}\,{\rm End}(T)^{\rm
  op}$ ($\geqslant 2$). If $d=2$ then ${\rm End}(T)^{\rm op}$ is
quasi-tilted, and hence there is nothing to prove. Assume that
$d>2$. Clearly, it suffices to show that there exists a tilting object
$T'\in\mathcal T$ obtained from $T$ by a tilting mutation and such
that ${\rm s.gl.dim.}\,{\rm End}(T')^{\rm op}=d-1$. Let $\mathcal
H\subseteq \mathcal T$ be a hereditary abelian category and
$\ell\geqslant 0$ be an integer like in
\ref{transjective}, \ref{position_AR5} or \ref{wild_sgldim} according
to whether (a), (b) or (c) (as stated at the beginning of the section)
holds true for $T$.

  \medskip
  
  In either case $T$ starts in $\mathcal H$ and ends in $\mathcal
  H[\ell]$. It follows from \ref{transjective}, \ref{tame_sgldim} and
  \ref{wild_sgldim} that $d=\ell+2$. Let
$T=T_1\oplus T_2$
be the direct sum decomposition such that $T_1\in \bigvee_{i=0}^{\ell-1} \mathcal
H[i]$ and $T_2\in\mathcal H[\ell]$. Let $T_2'\to M\to T_2\to
T_2'[1]$ be the triangle such that $M\to T_2$ is a right minimal ${\rm
  add}\,T_1$-approximation. Let
$T'=T_1\oplus T_2'$.
Clearly ${\rm Hom}(T_2,T_1)=0$. Thus $T'$ is a tilting object (\ref{setting_approx}).

\medskip

Applying \ref{positions_lem1} to $\mathcal A=\{0\}$ and $\mathcal
B=\mathcal H$ shows that $T'$ starts in $\mathcal H$ and ends in
$\mathcal H[\ell-1]$. In particular if $\ell=1$, that is $d=3$, then
$T'\in\mathcal H$; therefore ${\rm s.gl.dim.}\,{\rm End}(T')^{\rm
  op}=2=d-1$. From now on assume that $\ell\geqslant 2$. Therefore
$T'$ starts in $\mathcal H$ and ends in $\mathcal H[\ell -1]$, and $T$
and $T'$ have the same indecomposable direct summands lying in
$\mathcal H$ (\ref{positions_lem1}, with $\mathcal A=\{0\}$ and
$\mathcal B=\mathcal H$). The rest of the proof distinguishes three
cases according to situations (a), (b) and (c) listed earlier.

\medskip

(a) Because of the conditions satisfied by $T$ and $T'$,
the triple $(\Sigma',\mathcal H',\ell')$
arising from \ref{transjective} applied to $T'$ is such that
$\Sigma'=\Sigma$, $\mathcal H'=\mathcal H$, $\ell'=\ell-1$, and ${\rm
  s.gl.dim.}\,{\rm End}(T')^{\rm op}=d-1$.

\medskip

(b) Note that $\mathcal H$
arises from a weighted projective
      line
and, by assumption, $T$ ends in
$\mathcal H_0[\ell]$ (\ref{position_AR5}). Applying \ref{positions_lem1}
to $\mathcal A=\mathcal H_+$ and $\mathcal B=\mathcal H_0$ shows that
$T'$ starts in a {\tube} 
contained in $\mathcal H$ and ends in $\mathcal H_0[\ell -1]$. Hence,
${\rm s.gl.dim.}\,{\rm End}(T)^{\rm op}=\ell+1=d-1$ (\ref{tame_sgldim}).

\medskip

(c) Note that $\mathcal H$
arises either from a weighted projective line with
negative Euler
characteristic, or else from a hereditary algebra of
wild representation
type. Apply \ref{positions_lem1} in situation (2) (or, in situation
(4)) to the former case (or, to the latter case, respectively).
 Since $\ell\geqslant 2$ this shows that
$T'$ starts in a {\za} contained in $\mathcal H$, that it ends in $\mathcal
H[\ell-1]$ and that it does not end in a transjective
component. Therefore ${\rm s.gl.dim.}\,{\rm End}(T')^{\rm
  op}=\ell+1=d-1$ (\ref{wild_sgldim}).
\end{proof}

\appendix

\section{Morphisms in bounded derived categories of weighted
  projective lines}
\label{appendix_wpl}

This section collects some known results on hereditary abelian
categories arising from weighted projective lines,
and which are used
in the proof of the main theorems in this text. Also it proves some
useful technical facts on morphism spaces in the corresponding
bounded derived categories.  

Recall that like everywhere else in this text, by "{\hered} of $\mathcal{T}$" is meant a full subcategory  
${\mathcal H}'\subseteq \mathcal T$ which is hereditary abelian
and such that the embedding  
${\mathcal   H}'\hookrightarrow \mathcal T$ extends to a triangle
equivalence 
  $\mathcal D^b({\mathcal H}')\simeq \mathcal T$.

\subsection{Reminder on ${\rm coh}(\mathbb{X})$.} \label{intr}

Let $\mathbb{X}$ be a weighted projective line in the sense of
Geigle-Lenzing \cite{MR915180}. Let ${\mathcal H} ={\rm coh}({\mathbb
  X})$ the category of coherent sheaves. This section collects some
essential properties of ${\mathcal H}$ used in this text (see
\cite{MR2385175,MR1648603,MR1423314,MR2232010} and references therein
for details). The full subcategory of ${\mathcal H}$
  formed by objects of finite length in $\mathcal H$ (i.e. torsion
  sheaves on $\mathbb X$) is denoted by ${\mathcal H}_{0}$. The full
  subcategory of ${\mathcal H}$ of vector bundles (or, torsion-free
  sheaves) is denoted by ${\mathcal H}_{+}$. Note that 
  \begin{equation}
    \label{eq:1}
  {\rm Hom}({\mathcal H}_{0}, {\mathcal H}_{+}) = 0 \ {\rm and} \ {\mathcal
  H}_{+} = \{ X \in {\mathcal H}\ |\ {\rm Hom}( {\mathcal H}_{0}, X) = 0
\}\,.
\end{equation}

To each nonzero vector bundle $E$ is associated its rank ${\rm rk}(E) \in
{\mathbb N}\setminus \{0\}.$ \emph{Line bundles} are rank one vector
bundles. Given any vector bundle $E$, its rank is $r$ if and only if there exists a filtration $0 =
E_{0} \subset E_{1} \subset \cdots \subset E_{r} = E$ in $\mathcal H$ such that  $E_{i}/E_{i-1}$ is a line bundle for every $i$. 

The category ${\mathcal H}$ is abelian, hereditary,
and Krull-Schmidt; it
has Serre duality and Auslander-Reiten sequences. The main
specificities of the Auslander-Reiten structure of ${\mathcal H}$ used
in this text depend on the Euler characteristic
  $\chi(\mathbb X)$ (\cite[Sect. 10]{MR2385175} and
\cite[Sect. 4]{MR1648603}).
\begin{itemize}
\item
The indecomposable objects in $\mathcal H_0$ form a
  disjoint union of tubes in the Auslander-Reiten quiver of $\mathcal
  H$. This union is parametrised by $\mathbb X$.
\item
If $\chi(\mathbb X)>0$ then
  the indecomposable objects in ${\mathcal
    H}_{+}$ form a single Auslander-Reiten component of
 shape
  ${\mathbb Z} \Delta$ where $\Delta$ is a graph of extended Dynkin
  type.
\item 
If $\chi(\mathbb X)=0$ then
    ${\mathcal H}_{+}$ decomposes as
  ${\mathcal H}_{+} = \ \bigvee_{q \in {\mathbb Q}} {\mathcal H}^{(q)}
  $ where each ${\mathcal H}^{(q)}$ is a full subcategory of
  ${\mathcal H}$ isomorphic to ${\mathcal H}_{0}$. The subcategory
  ${\mathcal H}_{0}$ is also denoted by ${\mathcal H}^{(\infty)}$, and
  ${\rm Hom}({\mathcal H}^{(p)}, {\mathcal H}^{(q)}) = 0$ if $q < p
  \leq \infty$; in particular each ${\mathcal H}^{(p)}$ is a disjoint
  union of orthogonal tubes. For every $q\in \mathbb Q$, the
  category $\left(\bigvee_{r\in \mathbb Q\cup \{\infty\},\,q<r} \mathcal
  H^{(r)}[-1] \vee \bigvee_{s\in \mathbb Q,\,s\leqslant q} \mathcal
  H^{(s)}\right)$ is denoted by $\mathcal H\langle q\rangle$.
Recall that this is a hereditary abelian generating subcategory whose
subcategory of finite length objects is  $\mathcal H^{(q)}$.
Also,
$\mathcal H\langle \infty\rangle$ denotes $\mathcal H$. 
\item
 If $\chi(\mathbb X)<0$ then the indecomposable
      objects in $\mathcal H_+$ form
 a disjoint union of Auslander-Reiten components of
  shape ${\mathbb Z}A_{\infty}$  in the
    Auslander-Reiten quiver of $\mathcal H$.  
\end{itemize}

The following two properties on morphism spaces in ${\mathcal H}$
play a fundamental role in this text: 
\begin{itemize}
\item Let $L \in {\mathcal H}_{+}$ be a line bundle and let ${\mathcal
    U} \subseteq {\mathcal H}_{0}$ be a tube, then there exists a
  unique quasi-simple $S \in {\mathcal U}$ such that ${\rm Hom}(L,S)$
  is nonzero  (and, moreover,
   is one dimensional). 
\item In the wild type case, given nonzero vector bundles $E,F \in
  {\mathcal H}_{+}$  it exists an integer $n_{0}$ such that ${\rm
    Hom}(E, \tau^{n} F) \neq 0$ for every $n \geq n_{0}$.  
\end{itemize}

\subsection{Paths in $\mathcal D^b({\rm coh}(\mathbb
  X))$}\label{morph}
 Let $\mathbb X$ be a weighted projective line. Let ${\mathcal H} = {\rm
   coh}({\mathbb X})$. Let ${\mathcal T} = {\mathcal D}^{b}(\mathcal
 H)$.   The following result is  useful to investigate
 ${\rm s.gl.dim.}$ of endomorphism algebras of tilting objects in
 ${\mathcal T}$. 

\begin{lem*}
\begin{enumerate}[(1)]
\item  Let $E \in {\mathcal H}_{+}$ be  indecomposable. Let
  ${\mathcal U} \subseteq {\mathcal H}_{0}$ be a tube. Then there
  exists $S \in \mathcal U$ quasi-simple such that ${\rm Hom}(E,S)
  \neq 0$ and ${\rm Ext}^{1}(S, \tau E) \neq 0$. In particular ${\rm Hom}(E,
  {\mathcal U}) \neq 0$ and ${\rm Ext}^{1}({\mathcal U}, E) \neq
  0$. Moreover, if $S=X_0\to \cdots \to X_n\to \cdots$ is the unique
  infinite sectional path in the Auslander-Reiten quiver of $\mathcal
  U$ then ${\rm Hom}(E,X_n)\neq 0$ and ${\rm Ext}^1(X_n,\tau E)\neq 0$
  for every $n\geqslant 0$.
\item Let ${\mathcal U}, {\mathcal V} \subseteq {\mathcal H}_{0}$ be
  tubes. Let $j$ be a positive integer. Then there exist
  quasi-simples $S \in {\mathcal U}$ and $S' \in {\mathcal V}$
together with   indecomposable vector bundles $E, E' \in {\mathcal H}_{+}$ and
   a
  path of nonzero morphisms in ${\rm ind}\,\mathcal T $
\[
S \rightarrow E[1] \rightarrow \cdots \rightarrow E'[j] \rightarrow
S'[j]\,.
\] 
\end{enumerate}
\end{lem*}
\begin{proof}
$(1)$ Using the filtration of $E$ there exists a line bundle $L \in
{\mathcal H}_{+}$ and an epimorphism $E \rightarrow L$ in ${\mathcal
  H}$. Then there exists a  quasi-simple $S \in {\mathcal U}$
such that ${\rm Hom}(L,S) \neq 0$. Taking a composite morphism $E
\twoheadrightarrow 
L \rightarrow S$ shows that ${\rm Hom}(E,S) \neq 0$. Thus ${\rm Ext}^{1}(S,
\tau E) \neq 0$ because of Serre duality.  

\medskip

Let $E\to S$ be a nonzero morphism. For every $n\in\mathbb N$ let
$X_n\to X_{n+1}$ be an irreducible morphism. This is a
monomorphism in $\mathcal H$ because $\mathcal U$ is a tube.
This and
the path
$E\to S\to X_1\to \cdots \to X_n$ show that ${\rm Hom}(E,X_n)\neq 0$.
Serre duality entails ${\rm
  Ext}^1(X_n,\tau E)\neq 0$.

\medskip

$(2)$ The proof is an induction on $j$. Assume  that $j = 1$. Let
$L \in {\mathcal H}_{+}$ be any line bundle. Then there exist
quasi-simples $S \in {\mathcal U}, S' \in {\mathcal V}$ such that ${\rm
  Hom}(L, S') \neq 0$ and ${\rm Ext}^{1}(S, L) \neq 0$. Whence the path $S
\rightarrow L[1] \rightarrow S'[1]$. Assume that $j > 1$ and that
$(2)$ holds true for $j - 1$. By induction hypothesis and because of
the case $j = 1$, there exist paths  
$S \leadsto S''[1]$ and $S''[1] \leadsto S'[j]$ in ${\rm
  ind}\,{\mathcal T}$ for some quasi-simples
$S, S'' \in {\mathcal U}$ and $S' \in {\mathcal V}$. Whence a path $S
\leadsto S'[j]$. 
\end{proof}

\subsection{Morphisms between Auslander-Reiten components in
  ${\mathcal T}$} \label{morphAR}  

In order to understand better the {\hereds}
of ${\mathcal T}$ which contain indecomposable summands of a given
tilting object in $\mathcal T$, it is useful to know whether or
not there exists 
a nonzero morphism between two given Auslander-Reiten components in
${\mathcal T}$.  
Recall that every Auslander-Reiten component in ${\mathcal H}$ is stable and 
${\mathcal T} =  \bigvee_{i \in {\mathbb Z}} ({\mathcal H}_{+}[i] \vee
{\mathcal H}_{0}[i])$. Therefore any given Auslander-Reiten component
of ${\mathcal T}$
equals $\mathcal U[i]$ where $i\in\mathbb Z$ and
  $\mathcal U$ is either a tube in $\mathcal H_0$ or else consists of
  objects in $\mathcal H_+$.
Note that if ${\mathcal U},{\mathcal U'} \subseteq {\mathcal H}_{0}$
are distinct tubes then ${\rm Hom}({\mathcal U}, {\mathcal U}'[i]) =
0$ for every $i \in {\mathbb Z}$ because ${\mathcal U}$ and
${\mathcal U}'$ are orthogonal, because ${\mathcal H}$ is hereditary
and because of Serre duality.
These considerations
    together with
\ref{morph},  yield the following
proposition
where (2), (3) and (4) follow from (1).

\begin{prop*}
\begin{enumerate}[(1)]
\item Let ${\mathcal U} \subseteq {\mathcal H}_{0}$ and ${\mathcal V}
  \subseteq {\mathcal H}_{+}$ be Auslander-Reiten components. Let $i
  \in {\mathbb Z}$ then
\begin{itemize}
\item ${\rm Hom}({\mathcal U}, {\mathcal U}[i]) \neq 0
  \Leftrightarrow i = 0$ or $i = 1$,
\item ${\rm Hom}({\mathcal U}, {\mathcal V}[i]) \neq 0 \Leftrightarrow
  i =1$,
\item ${\rm Hom}({\mathcal V}, {\mathcal U}[i]) \neq 0 \Leftrightarrow
  i = 0$,
\item ${\rm Hom}({\mathcal V}, {\mathcal V}[i]) \neq 0 \Leftrightarrow
  i = 0$ or $i = 1$. 
\end{itemize}
\item Both ${\mathcal H}_{0}$ and ${\mathcal H}_{+}$ are convex in
  ${\mathcal T}$. 
\item Let ${\mathcal U} \subseteq {\mathcal H}_{0}$ be a tube. Let
  ${\mathcal V}$ be an Auslander-Reiten component of ${\mathcal T}$
  distinct from ${\mathcal U}[i]$ for every $i \in {\mathbb Z}$. Then:
${\mathcal V} \subseteq \bigvee_{\ell \in {\mathbb Z}} {\mathcal
  H}_{0}[\ell] \Leftrightarrow (\forall i \in {\mathbb Z}) \ {\rm
  Hom}({\mathcal U}, {\mathcal V}[i]) = 0 \Leftrightarrow (\forall i
\in {\mathbb Z}) \ {\rm Hom}({\mathcal V}, {\mathcal U}[i]) = 0$.
\item Let ${\mathcal V}$ be an Auslander-Reiten component of
  ${\mathcal T}$ then:
${\mathcal V} \subseteq {\mathcal H}_{+}
  \Leftrightarrow (\forall i \neq 0)\ \ \  {\rm Hom}({\mathcal V}, {\mathcal
    H}_{0}[i]) = 0$.
\end{enumerate}
\end{prop*}

\subsection{ A sufficient criterion for two {\hereds} to coincide} \label{suf}

The proofs of the main theorems use the following general fact on
tilting objects in ${\mathcal T}$: Let $T \in {\mathcal T}$ be a
tilting object such that 
$T  \in {\mathcal H}_{0} \vee {\mathcal H}_{+}[1] \vee {\mathcal
  H}_{0}[1]$ and $T$ has indecomposable summands both in ${\mathcal
    H}_{0}$ and ${\mathcal H}_{0}[1]$, then there is no {\hered} ${\mathcal H}' \subseteq {\mathcal T}$ such that
  $T \in {\mathcal H}'$ (\ref{positions_AR3}).
The proof of this
 fact is based on
the
  following result which gives a sufficient condition for two {\hereds} ${\mathcal H}, {\mathcal H}' \subseteq
  {\mathcal T}$ to coincide.  

\begin{prop*}
Let ${\mathcal H}' \subseteq {\mathcal T}$ be a {\hered} having no
nonzero projective object. Assume that there 
exists a tube ${\mathcal U} \subseteq {\mathcal H}_{0}$ such that
${\mathcal U} \subseteq {\mathcal H}_{0}'$. Then ${\mathcal H} =
{\mathcal H}'$. 
\end{prop*}
\begin{proof}
Note that since ${\mathcal H}'$ is not a module category it is the
category of coherent sheaves over some weighted projective line, and
since ${\mathcal D}^{b}({\mathcal H}') \simeq {\mathcal
  D}^{b}({\mathcal H})$ it follows that ${\mathcal H}'$ and ${\mathcal
  H}$ have isomorphic Auslander-Reiten quivers
\cite[Sect. 4]{MR1648603}. Hence ${\mathcal H}$ and ${\mathcal H}'$
play symmetric 
roles in the proposition.

It is useful to prove first that ${\mathcal H}_{0} = {\mathcal
  H}_{0}'$. Let ${\mathcal V} \subseteq {\mathcal H}_{0}$ be a tube
distinct from ${\mathcal U}$. Hence ${\mathcal V} \neq {\mathcal
  U}[j]$ for every $j \in {\mathbb Z}$. Applying \ref{morphAR} (part
$(3)$) to ${\mathcal H}, {\mathcal U}, {\mathcal V}$, yields
${\rm Hom}({\mathcal V}, {\mathcal U}[j]) = 0$ for every $j \in {\mathbb
  Z}$. Then applying the same result to ${\mathcal H}'$, ${\mathcal
  U}$, ${\mathcal V}$ entails that there exists $j \in {\mathbb Z}$
such that ${\mathcal V} \in {\mathcal H}_{0}'[j]$.  By absurd assume
that $j > 0$. Using \ref{morph} (part $(2)$) applied to ${\mathcal
  H}', {\mathcal U}$, ${\mathcal V}[-j]$ gives a path $S
\rightarrow E[1] \leadsto S'$ in ${\rm ind}\,{\mathcal T}$ such that $S \in
{\mathcal U}, S'\in {\mathcal V}$ are quasi-simple and $E \in
{\mathcal H}'_{+}$. But $S, S' \in {\mathcal H}_{0}$ and ${\mathcal
  H}_{0}$ is convex in ${\mathcal T}$. Therefore $E[1] \in {\mathcal
  H}_{0}$. Now because ${\rm ind}\,{\mathcal H}_{0}$ is a disjoint union of
pairwise orthogonal tubes and since $S \rightarrow E[1]$ is a nonzero
morphism with $S \in {\mathcal U} \subseteq {\mathcal H}_{0}$ and
$E[1] \in {\mathcal H}_{0}$, it follows that $E[1] \in {\mathcal U}$,
and hence $E[1] \in {\mathcal H}_{0}'$ (recall that ${\mathcal U}
\subseteq {\mathcal H}_{0}')$. This contradicts $E \in {\mathcal
  H}_{+}'$. Therefore $j \leq 0$. Dually, the same arguments show that
$j \geq 0$. Thus $j = 0$, and therefore ${\mathcal V} \subseteq
{\mathcal H}_{0}'$. This proves that ${\mathcal H}_{0} \subseteq
{\mathcal H}_{0}'$. And because of the symmetry between ${\mathcal H}$
and ${\mathcal H}'$ 
it follows that ${\mathcal H}_{0} = {\mathcal H}_{0}'$. 

There only remains
 to prove that ${\mathcal H}_{+} = {\mathcal H}_{+}'$. Let
${\mathcal V}$  be an Auslander-Reiten component of ${\mathcal T}$
such that ${\mathcal V} \subseteq {\mathcal H}_{+}$. Using
\ref{morphAR} (part $(4)$)  applied to ${\mathcal H}, {\mathcal U}$,
 ${\mathcal V}$ it follows that ${\rm Hom}({\mathcal V}, {\mathcal
  H}_{0}[i]) = 0 $ for $i \neq 0$. Since ${\mathcal H}_{0} = {\mathcal
  H}_{0}'$ the same result applied to ${\mathcal H}', {\mathcal U}$,
 ${\mathcal V}$ shows that $\mathcal V \subseteq {\mathcal
   H}_{+}'$. Therefore ${\mathcal H}_{+} \subseteq {\mathcal H}_{+}'$,
  and hence ${\mathcal H}_{+} = {\mathcal H}_{+}'$ by symmetry between
  ${\mathcal H}$ and ${\mathcal H}'$.
\end{proof}

\subsection{Description of one-parameter families of pairwise orthogonal tubes}
\label{tube_family}

\begin{prop*}
\begin{enumerate}[(1)]
\item The subcategory ${\mathcal H}_{0} \subseteq {\mathcal T}$ is a one-parameter family of tubes.
\item Let ${\mathcal T}_{0} \subseteq {\mathcal T}$ be a one-parameter
  family of tubes, then there exists an unique {\hered} ${\mathcal H}' \subseteq {\mathcal T}$ which is not a
  module category 
  and such that ${\mathcal T}_{0} = {\mathcal H}_{0}'$.
\end{enumerate}
\end{prop*}
\begin{proof}
(1)
   ${\mathcal H}_{0}$ is a disjoint union of pairwise
  orthogonal tubes (\ref{intr}) and it is convex in $\mathcal{T}$
  (\ref{morphAR}). Following \ref{morph} (part (2)), if a full convex
  subcategory intersects $\mathcal H_0$ and $\bigvee_{j\neq
    0}\mathcal H_0[j]$ nontrivially, then it intersects $\bigvee_{j\in
    \mathbb Z}\mathcal H_+[j]$ nontrivially. Therefore $\mathcal H_0$
  is maximal as a family of pairwise orthogonal tubes.

  (2) The uniqueness follows from \ref{suf}. Let ${\mathcal U}
  \subseteq {\mathcal T}_{0}$ be a tube. Let $\ell \in
  \mathbb Z$ be such that ${\mathcal U} \subseteq {\mathcal
    H}[\ell]$. If $\chi(\mathbb X)\neq 0$, then $\mathcal U \subseteq
  (\mathcal H[\ell])_0 = \mathcal H_0[\ell]$ (see \ref{intr}), and
  hence $\mathcal U\subseteq \mathcal H'_0$ where $\mathcal H'$
  denotes $\mathcal H[\ell]$. If $\chi(\mathbb X)=0$, then there
  exists $q\in \mathbb Q\cup \{\infty\}$ such that $\mathcal U\subseteq
  (\mathcal H[\ell])^{(q)}$, and hence $\mathcal U\subseteq \mathcal
  H'_0$ where $\mathcal H'$ denotes $\mathcal H[\ell]\langle
  q\rangle$.

  Let $\mathcal V\subseteq \mathcal T_0$ be a tube distinct from
  $\mathcal U$. By assumption on $\mathcal T_0$, there exists $t\in
  \mathbb Z$ such that $\mathcal V\subseteq \mathcal H'_0[t]$
  (~\ref{morphAR}, part (3)). Should
  $t$ be nonzero, $\mathcal T_0$ would intersect $\bigvee_{\ell\in
    \mathbb Z}\mathcal H'_+[\ell]$ because  it is convex in $\mathcal T$ (\ref{morph},
  part (2)). This is impossible because the 
  Auslander-Reiten components of $\mathcal T_0$ are pairwise
  orthogonal. Thus, $t=0$, and hence $\mathcal V\subseteq \mathcal
  H_0'$. Therefore, $\mathcal T_0\subseteq \mathcal H_0'$, and hence
  $\mathcal T_0=\mathcal H_0'$.
\end{proof}

\begin{center}
  {\textsc{Acknowledgements}}
\end{center}

  The work presented in this text was done while the second named
  author was associated professor  at
  Universit\'e Blaise Pascal and during visits to the first and third
  named authors. He would like to thank the members of
  the Department of Mathematics at Université Blaise Pascal as well as
  the first and third named authors for their warm hospitality.

\bibliographystyle{plain}
\bibliography{ALM-biblio}



\end{document}